\documentclass[10pt,twoside]{siamart1116}
\usepackage[english]{babel}
\usepackage{graphicx,epstopdf,epsfig}
\usepackage{amsfonts,epsfig,fancyhdr,graphics, hyperref,amsmath,amssymb}
\usepackage{color}
\usepackage{comment}
\usepackage{mathdots}% udots

\newcommand{\HE}{Name of Handling Editor}
\newcommand{\DoS}{Month/Day/Year}
\newcommand{\DoA}{Month/Day/Year}
\newcommand{\CA}{Vu Anh Le}

\newcommand{\Names}{Hieu V. Ha, Vu A. Le, Tu T. C. Nguyen and Hoa D. Quang}
\newcommand{\Title}{Classification of Solvable Lie algebras whose non-trivial Coadjoint Orbits of simply connected Lie groups are all of Codimension 2}

\newtheorem{remark}[theorem]{Remark}

\newcommand{\R}{\mathbb{R}}
\newcommand{\C}{\mathbb{C}}
\newcommand{\Gc}{\mathcal G}
\newcommand{\Hc}{\mathcal H}

\DeclareMathOperator*{\rank}{rank}

\newcommand{\MD}[2]{\text{MD}_{#2}(#1)}
\newcommand{\ad}[1]{\textnormal{ad}_{#1}}
\newcommand{\adg}[1]{\textnormal{ad}^1_{#1}}
\newcommand{\adgg}[1]{\textnormal{ad}^2_{#1}}
 
\renewcommand{\theequation}{\thesection.\arabic{equation}}
\renewtheorem{theorem}{Theorem}[section]

\begin{document}

\bibliographystyle{plain}

%  Leave these commented lines here
%\input{ELAheader-template.tex}
% ELA insert correct page number
\setcounter{page}{1}

\thispagestyle{empty}

%Insert the title of the paper
 \title{\Title\thanks{Received
 by the editors on \DoS.
 Accepted for publication on \DoA. 
 Handling Editor: \HE. Corresponding Author: \CA}}

\author{
    Hieu Van Ha\thanks{University of Economics and Law, Ho Chi Minh City, Vietnam and Vietnam National University, Ho Chi Minh City, Vietnam (hieuhv@uel.edu.vn) } 
   \and
    Vu Anh Le\thanks{University of Economics and Law, Ho Chi Minh City, Vietnam and Vietnam National University, Ho Chi Minh City, Vietnam (vula@uel.edu.vn) }%\footnotemark[2]
    \and 
    Tu Thi Cam Nguyen\thanks{University of Science - Vietnam National University Ho Chi Minh City, Vietnam and  Can Tho University, Can Tho, Vietnam (camtu@ctu.edu.vn)}
    \and
    Hoa Duong Quang\thanks{Hoa Sen University, Vietnam (hoa.duongquang@hoasen.edu.vn).}
}

\markboth{\Names}{\Title}

\maketitle

\begin{abstract}
    We give a classification of real solvable Lie algebras whose non-trivial coadjoint orbits of corresponding simply connected Lie groups are all of codimension 2. These Lie algebras belong to a well-known class, called the class of MD-algebras.
\end{abstract}


%%%%%%%%%%%%%%%%%%%%%%%%%%%%%%%%%%%%%%%%%%%%%%%%%%%%%%%%%%%%%

\section{Introduction}
    The problem of the classification of Lie algebras (as well as Lie groups) has received much attentions since the early 20$^\textnormal{th}$ century. However, this is still an open problem. By Levi's decomposition and the Cartan's theorem, we know that the problem of classification of Lie algebras over any field of characteristic zero are reduced to the problem of classification of solvable ones. However, until now, there is no a complete classification of $n$ dimensional solvable Lie algebras if $n\geq 7$. And this classification problem seems to be impossible to solve, unless there is a suitable change on the definition of term ``classification" or there is a completely new method to classify those Lie algebras \cite{Boza13}. 
    
	As we know, the Lie algebra of a (simply connected) Lie group is commutative if and only if all of its coadjoint orbits are trivial (or of dimension 0). However, Lie groups which has a non-trivial coadjoint orbit are much more complicated. In 1980, while searching for the class of Lie groups whose $C^*$-algebra can be characterized by BDF $K$-functions, Do Ngoc Diep proposed to study a class of Lie groups whose non-trivial coadjoint orbits have the same dimension \cite{Do99}. He named this class as MD-class. Any Lie group belongs to this class is called an MD-group and the Lie algebra of any MD-group is called an MD-algebra. 
	
	It can be said that Vuong Manh Son and Ho Huu Viet were the authors who faced the problem of classification MD-algebras (as well as MD-groups) firstly. In 1984, they gave not only the classification of MD-groups whose non-trivial coadjoint orbits are of the same dimension as the group but also some important characteristics of this class. For example, they showed that any non-commutative MD-algebra is either 1-step solvable or 2-step solvable, i.e. the second derived algebra is commutative \cite{Son-Viet}. Afterward, from 1990, Vu A. L. and Hieu V. H. (the authors of this paper) gave the classification (up to isomorphic) of some subclasses; including all MD-algebras of dimension 4  \cite{Le90-2}, all MD-algebras of dimension 5 \cite{Vu-Shum, LHT11}, all MD-algebras which have the first derived ideal of dimension 1 or codimension 1 \cite{Vu-MD}. 
	
	Besides, a list of all simply connected Lie groups whose coadjoint orbits are of dimension up to 2 was given by D. Arnal et al. in 1995 \cite{Arnal}. In 2019, Michel  Goze  and  Elisabeth  Remm used Cartan class to give the classification of all Lie algebras that all non-trivial coadjoint orbits of corresponding Lie groups are of dimension 4 \cite{GR19}. Remark that the Lie algebras classified in \cite{Arnal} and \cite{GR19} are all MD-algebras in terms of Diep. Moreover,  Goze  and  Remm also gave some characteristics of the class of MD-algebras whose non-trivial coadjoint orbits are of codimension 1. Recently, in an earlier article \cite{HHV21}, we have classified all real solvable Lie algebras whose non-trivial coadjoint orbits are of codimension 1. Now, we will give the complete classification of real solvable Lie algebras whose non-trivial coadjoint orbits are of codimension 2.  
	
	The paper is organized into 6 sections, including this introduction. In Section \ref{section-pre}, we will recall some basic preliminary concepts, notations and properties which will be used throughout the paper. In Section \ref{section-3} and Section \ref{section-4}, we will give the classification of 1-step solvable Lie algebras whose non-trivial coadjoint orbits are of codimension 2 [Theorem \ref{classification-commutative3}, Theorem \ref{theorem62}]. In Section \ref{section-5}, we will study the case of such 2-step solvable Lie algebras [Theorem \ref{main-noncommutative}], and complete the results in Sections \ref{section-3}, \ref{section-4}. Tables containing a list of results are provided in the last section.
    
% ---------- PRELIMINARIES --------------------------
% ---------------------------------------------------
\section{Preliminaries}\label{section-pre}
    We now introduce some key definitions, notations and terminologies.  For more details, we refer reader to \cite{Kiri}.
    
    \begin{itemize}
        \item Throughout this paper, the underlying field is always the field $\R$ of real numbers and $n$ is an integer $\geq 2$ unless otherwise stated.
       	\item For any Lie algebra $\Gc$ and  $0 < k \in \mathbb N$, the direct sum $\Gc \oplus \R ^k$ is called a trivial extension of $\Gc$.
        \item A Lie algebra $(\Gc,[\cdot,\cdot])$ is said to be $i$-step solvable or solvable of degree $i$ if its $i$-th derived algebra $\Gc^i:=[\Gc^{i-1},\Gc^{i-1}]$ is commutative and non-trivial (i.e. $\neq \{0\}$) where $\Gc^0:=\Gc$ and $0<i\in \mathbb N$.
    	\item An $n\times n$ matrix whose $(i,j)$-entry is $a_{ij}$ will be written as $(a_{ij})_{n\times n}$. While the $(i,j)$-entry of a matrix $A$ will be denoted by $(A)_{ij}$. The transpose of $A$ will be denoted by $A^t$. For an endomorphism $f$ on a vector space $V$ of dimension $n$, the matrix of $f$ with respect to a basis $\mathfrak{b}:=\{x_1,\dots,x_n\}$ of $V$ will be denoted by $[f]_{\mathfrak{b}}$. For short, if $U:=\langle x_k, \dots,x_n\rangle$ is the subspace of $V$ spanned by $\{x_k, \dots, x_n\}$ and if $g:U \rightarrow U$ is a linear endomorphism on $U$ then the notation $[g]_\mathfrak{b}$ will be used to denote the matrix of $g$ with respect to the basis $\{x_k, \dots,x_n\}$ of $U$.
    	\item As usual, the dual space of $V$ will be denoted by $V^*$. It is well-know that if $\{x_1,x_2,\dots,x_n\}$ is a basis of $V$ then $\{x_1^*,\dots,x_n^*\}$ is a basis of $V^*$, where each $x_i^*$ is defined by $x_i^*(x_j)=\delta_{ij}$ (the Kronecker delta symbol)  for $1 \leq i, j \leq n$.
    	\item For any $x \in \Gc$, we will denote by $\ad{x}$ the adjoint action of $x$ on $\Gc$, i.e. $\ad{x}$ is the endomorphism on $\Gc$ defined by $\ad{x}(y)= [x, y]$ for every $y \in \Gc$. By $\adg{x}$ and $\adgg{x}$, we mean the restricted maps of $\ad{x}$ on $\Gc^1$ and $\Gc^2$, respectively. Since $\Gc^1$ and $\Gc^2$ are ideals of $\Gc$, $\adg{x}$ and $\adgg{x}$ will be treated as endomorphisms on $\Gc^1$ and $\Gc^2$, respectively. 
    	\item In this paper, we will use the symbol $I$ to denote the $2\times 2$ identity matrix, and use $J$ to denote the following $2\times 2$ matrix $\begin{bmatrix}
    	    0 & 1 \\ -1 & 0
    	\end{bmatrix}.$ We shall denote by $\boldsymbol{0}$ the zero matrix of suitable size.
    \end{itemize}

    \begin{definition}{\rm
        Let $G$ be a Lie group and let $\Gc$ be its Lie algebra. If $\text{Ad}:G \rightarrow  \text{Aut}(\Gc)$ denotes the {\em adjoint representation} of $G$. Then the action
        \begin{align*}
            K: & \quad G \rightarrow \text{Aut}(\Gc^*) \\
               & \quad g \mapsto K_g
        \end{align*}
        defined by
        \begin{equation*}
            K_g(F)(x)=F(\text{Ad}(g^{-1})(x)) \text{ for } F\in\Gc^*, x\in\Gc.
        \end{equation*}
        is called the {\em coadjoint representation} of $G$ in $\Gc^*$. Each orbit of the coadjoint representation of $G$ is called a {\em coadjoint orbit}, or a {\em K-orbit} of $G$.}
    \end{definition}
    
    For each $F\in\Gc^*$, the coadjoint orbit for $F$ is denoted by $\Omega_F$, i.e. 
    \begin{equation*}
        \Omega_F=\{K_g(F): g\in G\}.
    \end{equation*}
    The dimension of each coadjoint orbit is determined via the following proposition.
    \begin{proposition}\cite{Kiri} \label{rankbf} Let $F$ be any element in $\Gc^*$. If $\{x_1,x_2,\dots,x_n\}$ is a basis of $\Gc$ then 
        \begin{equation*}
            \dim \Omega_F=\rank{\bigl(F([x_i,x_j])\bigr)_{n\times n}}.
        \end{equation*}
    \end{proposition}
    \begin{remark}\label{remark23}
        The dimension of each K-orbit $\Omega_F$ is always even for every $F\in\Gc^*$. Moreover, $\dim\Omega_F>0$ if and only if $F|_{\Gc^1}\neq 0$.
    \end{remark}
    
    As mentioned in previous section, this paper concerns with Lie algebras whose non-trivial coadjoint orbits are all of the same dimension. 
    \begin{definition}\cite{Do99,Son-Viet} {\rm 
        An {\em MD-group} is a finite-dimensional, simply connected and solvable Lie group whose non-trivial coadjoint orbits are of the same dimension. The Lie algebra of an MD-group is called an {\em MD-algebra}. An MD-algebra $\Gc$ is called an MD$_k(n)$-algebra if $\dim\Gc=n$ and the same dimension of non-trivial coadjoint orbits is equal to $k$.}
    \end{definition}
    
     One of the most interesting characteristics on this class is about the degree of solvability which is proven by Son \& Viet \cite{Son-Viet}.
    \begin{proposition}\cite{Son-Viet}\label{solvability2}
        If $\Gc$ is an MD-algebra then the degree of solvability is at most 2, i.e. $\Gc^3=\{0\}$.
    \end{proposition}
    
    Therefore, the problem of classification of MD-algebras falls naturally into two parts:
    (1) the classification of 1-step solvable ones, and (2) the classification of 2-step solvable ones. However, if $\Gc$ is a 2-step solvable MD-algebra then $\Gc/\Gc^2$ is a 1-step solvable MD-algebra \cite[Theorem 3.5]{HHV21}. Hence, we should firstly study some interesting properties of 1-step solvable MD-algebras.
    
    \begin{proposition}\cite{HHV21}\label{pro211}
        Let $\Gc$ be a 1-step solvable Lie algebra of dimension $n$ such that its non-trivial coadjoint orbits are all of codimension $k$. If $\dim \Gc^1\geq n-k+1$ then
        %    \begin{enumerate}
        %        \item There is $x$ in $\Gc$ so that $\adg{x}$ is a linear automorphism on $\Gc^1$.
        %        \item 
                $\Gc$ is isomorphic to the semi-direct product $\mathcal{L}\oplus_{\rho} \Gc^1$ where $\mathcal{L}$ is a commutative sub-algebra of $\Gc$ and $\rho$ is defined by 
                \begin{equation}\label{def-rho}
                    \begin{array}{llll}
                        \rho: & \mathcal{L}\times\Gc^1 & \rightarrow & \Gc^1 \\
                         & (x,y) & \mapsto & [x,y].
                    \end{array}
                \end{equation}
       %     \end{enumerate}
    \end{proposition}
    
    % Furthermore, for every $x,y\in \Gc$, the Jacobi identity gives us
    % \begin{equation}\label{eq-commutative}
    %     \ad{x}\ad{y}-\ad{y}\ad{x}=\ad{[x,y]}.
    % \end{equation}
    % Since $\text{trace}(AB-BA)=0$ for every two square matrices $(A, B)$ of the same size, we obtain 
    % \[
    %     \text{trace}(\ad{x})=\text{trace}(\adg{x})=\text{trace}(\adgg{x})=0 \quad \forall x\in\Gc^1.
    % \] 
    
    Moreover, if $\Gc$ is 1-step solvable then $[[x,y],z]=0$ for every $x,y\in\Gc, z\in\Gc^1$. It follows immediately from the Jacobi identity that   $\adg{x}\adg{y}=\adg{y}\adg{x}$ for every $x,y\in\Gc$. 
    \begin{lemma}\label{1-stepcomtu}
        If $\Gc$ is 1-step solvable then $\{\adg{x}:x\in\Gc\}$ is a family of commuting endomorphisms. 
    \end{lemma}
    % Similarly, if $\Gc$ is 2-step solvable then $\adgg{x}\adgg{y}=\adgg{y}\adgg{x}$ for every $x,y\in\Gc$. This follows immediately from Proposition \ref{solvability2} that $\adgg{x}\adgg{y}=\adgg{y}\adgg{x}$ for every $x,y$ in an MD-algebra. In summary, we have proven the following proposition. 
    % % 
    % \begin{proposition}\label{commutative} 
    %     Let $\Gc$ be a finite dimensional Lie algebra. Then
    %     \begin{enumerate}
    %         \item $\textnormal{trace}(\ad{x})=\textnormal{trace}(\adg{x})=\textnormal{trace}(\adgg{x}) =0 $ for all $x\in\Gc^1$.
    %         \item $\{\adg{x}:x\in\Gc\}$ is a commuting family of endomorphisms on $\Gc^1$ provided $\Gc$ is 1-step solvable (i.e. $\Gc^1$ is commutative).
    %         \item $\{\adgg{x}:x\in\Gc\}$ is a commuting family of endomorphisms on $\Gc^2$ provided $\Gc$ is an MD-algebra.
    %     \end{enumerate}
    % \end{proposition}
    
    It is well-known that an arbitrary set of commuting  matrices over an algebraic closed field may be simultaneously brought to triangular form by a unitary similarity \cite{Morris,Heydar}. A similar version for the case of the real field is given in the following proposition.
    \begin{proposition}\label{triangular}
        Let $\mathcal{S}$ be a set of commuting real matrices of the same size. Then $\mathcal{S}$ is block simultaneously triangularizable in which the maximal size of each block is 2. In other words, there is a non-singular real matrix $T$ so that
        \begin{equation*}
            T\mathcal{S}T^{-1} =\left[\begin{array}{llllllll}
                *_{2\times 2} & \\
                  & \ddots & & & \textnormal{\huge *}\\
                  &        & *_{2\times 2} \\
                  &        &    & * \\
                  &\textnormal{\huge 0}&  &  &  \ddots \\
                  & & & & &  *
            \end{array}\right]
        \end{equation*}
        where each block $*_{2\times 2}$ is of the form $\begin{bmatrix}
            a & b\\
            -b & a
        \end{bmatrix}$ for some $a,b\in\R$ ($b$ is not necessary to be non-zero). 
    \end{proposition}
    
    The following lemma is a straightforward but useful consequence of Propositions \ref{pro211}, \ref{triangular} and Lemma \ref{1-stepcomtu}.
    \begin{lemma}\label{cor2}
        Let $\Gc$ be a 1-step solvable $\MD{n}{n-2}$-algebra such that $m:=\dim\Gc^1$ is strictly greater than 2. Then there is a basis $\mathfrak{b}:=\{x_1,\ldots,x_n\}$ of $\Gc$ so that
            \begin{itemize}
                \item $\Gc^1=\langle x_{n-m+1},\ldots, x_n\rangle$ is commutative,
                \item $[x_i,x_j]=0$ for every $1\leq i,j\leq n-m$,
                \item The matrices $[\adg{x_1}]_\mathfrak{b}, [\adg{x_2}]_\mathfrak{b},\ldots, [\adg{x_{n-m}}]_\mathfrak{b}$ are of the block triangular form in the sense of Proposition \ref{triangular}.
    %            \item $\adg{x_1}$ is a linear automorphism on $\Gc^1$.
            \end{itemize}
    \end{lemma}
    \begin{remark}\label{remark210}
        In the above lemma, we can choose $\mathfrak{b}$ so that the space $\mathcal{L}$ in the semi-direct sum $\mathcal{L}\oplus_\rho\Gc^1$ of $\Gc$ is spanned by $\{x_1, \dots, x_{n-m}\}$. If so, for each $F\in\Gc^*$, 
        \begin{equation*}
            \left(F\left([x_i,x_j]\right)\right)_{n\times n} = \begin{bmatrix}
                \boldsymbol 0 & P_F \\
                -P_F^t & \boldsymbol 0
            \end{bmatrix},
        \end{equation*}
        where $P_F$ is an $(n-m)\times m$ matrix which is defined by:
        \begin{equation*}
            (P_F)_{ij}:= F\left([x_i,x_{n-m+j}]\right).
        \end{equation*}
        By Proposition \ref{rankbf}, 
        \begin{equation*}
            \dim \Omega_F = 2 \rank{(P_F)} \quad \text{for every } F\in\Gc^*.
        \end{equation*}
    \end{remark}
    
%     By applying Proposition \ref{rankbf} and Remark \ref{remark210}, we can prove easily the following lemma.
%     \begin{lemma}\label{cor1}
%         Let $\Gc$ be a 1-step solvable Lie algebra of dimension $n$ which can be expressed as the semi-direct sum of $\Gc^1$ and a commutative sub-algebra $\mathcal{L}$ of $\Gc$, i.e. $\Gc=\mathcal{L}\oplus_{\rho}\Gc^1$, where $\rho$ is defined in the equation (\ref{def-rho}). And let $F$ be any element in $(\Gc)^*$. Then 
%         \begin{equation}
%             \dim \Omega_F = 2 \rank{(P_F)}
%         \end{equation}
%         where $P_F$ is defined in the equation (\ref{Pf}).
%     \end{lemma}
% %    
    
Finally, if $\Gc$ is an $\MD{n}{n-2}$-algebra then $\Gc/\Gc^2$ is an $\MD{n-\dim\Gc^2}{n-2}$-algebra \cite[Theorem 3.5]{HHV21}. Hence, we should recall here the classifications of $\MD{n}{n-1}$-algebras and $\MD{n}{n}$-algebras which are solved by Hieu et. al. \cite{HHV21} and Son \& Viet \cite{Son-Viet}, respectively. 
    
    \begin{proposition}\cite{HHV21}\label{pro210}
        Let $\Gc$ be a real MD$_{n-1}(n)$-algebra with $n\geq 5$. Then $\Gc$ is isomorphic to one of the followings: 
            \begin{enumerate}
                \item A trivial extension of $\textnormal{aff}(\mathbb C)$, namely $\mathbb R \oplus \textnormal{aff}(\mathbb C)$, where $\textnormal{aff}(\mathbb C):=\langle x_1,x_2,y_1,y_2\rangle$ is the complex affine algebra defined by
                \begin{equation*}
                    [x_1,y_1]=y_1, [x_1,y_2]=y_2, [x_2,y_1]=-y_2,[x_2,y_2]=y_1. 
                \end{equation*}
                \item The real Heisenberg Lie algebra \[\mathfrak{h}_{2m+1}:=\langle x_i,y_i,z:i=1,\dots,m\rangle, \] 
                with $[x_i,y_i]=z$ for every $1\leq i\leq m$.
                    
                \item The Lie algebra \[\mathfrak{s}_{5,45}:=\langle x_1,x_2,y_1,y_2,z\rangle,\] with
                    \[
                        [x_1,y_1]=y_1, [x_1,y_2]=y_2, [x_1,z]=2z, [x_2,y_1]=y_2, [x_2,y_2]=-y_1, [y_1,y_2]=z.
                    \]
            \end{enumerate}
    \end{proposition}

    \begin{proposition}\cite{Son-Viet}\label{sonviet}
        Let $\mathcal{G}$ be a real $\MD{n}{n}$-algebra. Then $\mathcal{G}$ is isomorphic to one of the following forms:
        \begin{enumerate}
            \item The real affine algebra $\textnormal{aff}(\mathbb{R}):=\langle x,y\rangle$ with 
                $
                    [x,y]=y,
                $
            \item The complex affine algebra $\textnormal{aff}(\mathbb{C})$ defined in Proposition \ref{pro210}.
         \end{enumerate}
    \end{proposition}   
    \begin{remark}\label{remark213}
  Note that the dimension of any coadjoint orbit is even [Remark \ref{remark23}], therefore if $\Gc$ is an $\MD{n}{n-2}$-algebra then $n$ must be even. The case $n=2$ is trivial. The case $n=4$ is solved completely in \cite{Le90-2}. Namely, up to an isomorphism, in the $\MD{4}{2}$-class there are 5 decomposable algebras and 8 indecomposable ones as follows: %an $\MD{4}{2}$-algebra is isomorphic to one of the following algebras:
    \begin{enumerate}
        \item[(1)] The decomposable case:
        \begin{itemize}
            \item[(i)] $\textnormal{aff}(\R)\oplus \R^2$.
            \item[(ii)] $\mathfrak{s}_{3}\oplus \R$ where $\mathfrak{s}_3 \in \{\mathfrak{n}_{3,1},\, \mathfrak{s}_{3,1},\, \mathfrak{s}_{3,2},\,
        \mathfrak{s}_{3,3}\}$, i.e. $\mathfrak{s}_{3}$ is a  non-commutative solvable Lie algebra of dimension 3 according to the notation of \cite{Snob}.
        \end{itemize}
         % where $\textnormal{aff}(\R)$ is the real affine Lie algebra of dimension 2.
        %\item[(2)] : %, there are 4 such non-isomorphic families: $\mathfrak{n}_{3,1}$, $\mathfrak{s}_{3,1}$, $\mathfrak{s}_{3,2}$, and $\mathfrak{s}_{3,3}$. 
        \item[(2)] The indecomposable case: %One of the following 8 non-isomorphic families of indecomposable Lie algebras of dimension 4:
        $\mathfrak{n}_{4,1}$, $\mathfrak{s}_{4,1}$, $\mathfrak{s}_{4,2}$, $\mathfrak{s}_{4,3}$, $\mathfrak{s}_{4,4}$, $\mathfrak{s}_{4,5}$, $\mathfrak{s}_{4,6}$, $\mathfrak{s}_{4,7}$ according to the notation of \cite{Snob}.
    \end{enumerate}
    \end{remark}
   
    Hence, to completely classify the $\MD{n}{n-2}$-class, we only have to consider the remaining case when $n\geq 6$.

\section[One step solvable MD{n-2}{n}-algebras]{One-step solvable $\boldsymbol{\MD{n}{n-2}}$-algebras} \label{section-3}
    
% Theorem 1 
    %This section presents main results of the paper. We firstly classify 1-step solvable $\MD{n}{n-2}$ Lie algebras whose first derived algebras are of dimension at least 3 [Theorem \ref{classification-commutative3}].
    %Afterward, we will classify $\MD{n}{n-2}$-algebras whose derived algebras are of dimension at most 2 [Theorem \ref{theorem62}]. 
    %Finally, we will use the classification of $\MD{n}{n-1}$-algebras in Proposition \ref{pro210} to obtain the classification of 2-step solvable $\MD{n}{n-2}$-algebras [Theorem \ref{main-noncommutative}].
    According to Proposition \ref{solvability2} and Lemma \ref{cor2}, the classification of $\MD{n}{n-2}$-algebras falls naturally into three problems:
    \begin{itemize}
        \item The problem of classification those 1-step solvable algebras which have the derived algebra of dimension at least 3.
        \item The problem of classification of those 1-step solvable algebras which have the derived algebra of dimension at most 2.
        \item The problem of classification of those 2-step solvable algebras.
    \end{itemize}
    We will solve the first item in this section. The remaining items will be solved in the next sections.

    \begin{theorem}\label{classification-commutative3}
        Let $\Gc$ be a 1-step solvable $\MD{n}{n-2}$-algebra of dimension $n\geq 6$ and $\dim\Gc^1\geq 3$. Then $n$ must be 6 and $\Gc$ is isomorphic to one of the following families: $\mathfrak{s}_{6,211}$, $\mathfrak{s}_{6,225}$, $\mathfrak{s}_{6,226}$, $\mathfrak{s}_{6,228}$\footnote{Some algebras contained in families listed in \cite{Snob} are not MD-algebras, we will give the detail Lie brackets of these Lie algebras (which are MD-algebras) in the final section} listed in \cite{Snob}.

    \end{theorem}
    
    \begin{remark}
        If $\Gc$ is a decomposable $\MD{6}{4}$-algebra then $\Gc$ is a trivial extension of either an %isomorphic to $\R^i\oplus \Gc'$ where either $\Gc'$ is an 
        indecomposable $\MD{5}{4}$-algebra %(and $i=1$) 
        or an indecomposable $\MD{4}{4}$-algebra %(and $i=2$)
        \cite[Theorem 3.1]{HHV21}. These indecomposable MD-algebras are classified in \cite{Son-Viet,LHT11,Vu-Shum}. Based on their classification, there are exactly one indecomposable $\MD{4}{4}$-algebra $\textnormal{aff}(\C)$ and exactly one indecomposable $\MD{5}{4}$-algebra %which are  and 
        $\mathfrak{s}_{5,45}$ in Proposition \ref{pro210} %respectively. 
        Hence, if $\Gc$ is a decomposable $\MD{6}{4}$-algebra then $\Gc$ is either isomorphic to $\R^2\oplus\textnormal{aff}(\C)$ or isomorphic to $\R\oplus\mathfrak{s}_{5,45}$.
    \end{remark}

    In order to prove Theorem \ref{classification-commutative3}, we will need the following lemma.
% Lemma 4.2. ---------------------------------------    
    \begin{lemma}\label{lemma35}
        Let $f, g$ be two commutative endomorphisms on $\R^4$, i.e. $f\circ g=g\circ f$. Assume that the matrices of $f$ and $g$ with respect to a basis $\mathfrak{b}$ are equal to
        \begin{equation*}
            [f]_{\mathfrak{b}}=\begin{bmatrix}
                A_1 & A_2 \\
                \boldsymbol{0} & I
            \end{bmatrix}, [g]_{\mathfrak{b}}=\begin{bmatrix}
                B_1 & B_2 \\
                \boldsymbol{0} & J
            \end{bmatrix};
        \end{equation*}
        where $A_1,A_2, B_1,B_2$ are $2\times 2$ matrices. If either $\det(B_1^2+I)\neq 0$ or $\det(A_1-I)\neq 0$ then there is a basis $\mathfrak{b}'$ of $\R^4$ so that 
        \begin{equation*}
            [f]_{\mathfrak{b}'}=\begin{bmatrix}
                A_1 & \boldsymbol{0} \\
                \boldsymbol{0} & I
            \end{bmatrix}, [g]_{\mathfrak{b}'}=\begin{bmatrix}
                B_1 & \boldsymbol{0} \\
                \boldsymbol{0} & J
            \end{bmatrix}.
        \end{equation*}
    \end{lemma}
    
    \begin{proof}[Proof of Lemma \ref{lemma35}]
        Let's denote the vectors in the basis $\mathfrak{b}$ by $\{y_1,y_2,y_3,y_4\}$.
        \begin{itemize}
            \item If $\det(B_1^2+I)\neq 0$, then we first claim that there are $\alpha,\beta,\gamma,\delta \in\R$ so that
                \begin{equation*}
                    \begin{bmatrix}
                        -\gamma & \alpha \\
                        -\delta & \beta
                    \end{bmatrix} =
                    B_2+B_1\begin{bmatrix}
                        \alpha & \gamma \\
                        \beta & \delta
                    \end{bmatrix}.
                \end{equation*}
                Indeed, the above system is equivalent to 
                \begin{equation*}
                    \left\{
                        \begin{array}{rl}
                            \begin{bmatrix}
                            -\gamma \\ -\delta
                            \end{bmatrix} = & 
                            \begin{bmatrix}
                                (B_2)_{11} \\ (B_2)_{21}
                            \end{bmatrix} + B_1\begin{bmatrix}
                                \alpha \\ \beta
                            \end{bmatrix} \\
                            \begin{bmatrix}
                            \alpha \\ \beta
                            \end{bmatrix} = & 
                            \begin{bmatrix}
                                (B_2)_{12} \\ (B_2)_{22}
                            \end{bmatrix} + B_1\begin{bmatrix}
                                \gamma \\ \delta
                            \end{bmatrix}
                        \end{array}
                    \right.,
                \end{equation*}
            or
            \begin{equation*}\label{eq425}
                \left\{
                    \begin{array}{rl}
                        \begin{bmatrix}
                        -\gamma \\ -\delta
                        \end{bmatrix} = & 
                        \begin{bmatrix}
                            (B_2)_{11} \\ (B_2)_{21}
                        \end{bmatrix} + B_1\begin{bmatrix}
                            \alpha \\ \beta
                        \end{bmatrix} \\
                        (B_1^2+I)\begin{bmatrix}
                        \alpha \\ \beta
                        \end{bmatrix} = & 
                        \begin{bmatrix}
                            (B_2)_{12} \\ (B_2)_{22}
                        \end{bmatrix} - B_1\begin{bmatrix}
                            (B_2)_{11} \\ (B_2)_{21}
                        \end{bmatrix}
                    \end{array}
                \right..
            \end{equation*}
            The existence of $\alpha,\beta,\gamma,\delta$ follows from the non-singularity of $B_1^2+I.$ 
            
            Let $\mathfrak{b}':=\{y_1',y_2',y_3',y_4'\}$ be a basis of $\R^4$ defined by:
            \begin{equation*}
                \left\{
                    \begin{array}{ll}
                        y_1'=y_1, y_2'=y_2\\
                        y_3'=y_3+\alpha y_1+\beta y_2\\
                        y_4'=y_4+\gamma y_1+\delta y_2.
                    \end{array}
                \right.
            \end{equation*}
            Then the matrix of $f$ and $g$ with respect to $\mathfrak{b}'$ are determined as
            \begin{equation*}
                [f]_{\mathfrak{b}'}=\begin{bmatrix}
                    A_1 & A_2' \\
                    \boldsymbol{0} & I
                \end{bmatrix}, \quad [g]_{\mathfrak{b}'}=\begin{bmatrix}
                    B_1 & \boldsymbol{0} \\
                    \boldsymbol{0} & J
                \end{bmatrix},
            \end{equation*}
            for some $2\times 2$ matrix $A_2'$.
            Moreover,      
                \begin{equation*}
                    f\circ g=g\circ f \Longleftrightarrow A_2' \times  J=B_1\times A_2' 
                        \Longleftrightarrow 
                        \left\{
                            \begin{array}{rl}
                                    -\begin{bmatrix}
                                    (A'_2)_{12} \\ (A'_2)_{22}
                                \end{bmatrix}= B_1 \begin{bmatrix}
                                    (A'_2)_{11}  \\ (A'_2)_{21}
                                \end{bmatrix} \\
                                \begin{bmatrix}
                                    (A'_2)_{11} \\ (A'_2)_{21}
                                \end{bmatrix}= B_1 \begin{bmatrix}
                                    (A'_2)_{12}  \\ (A'_2)_{22}
                                \end{bmatrix}
                            \end{array}
                        \right..
                \end{equation*}
                Hence,
                \begin{equation*}
                    \left\{
                        \begin{array}{rl}
                                -\begin{bmatrix}
                                (A'_2)_{12} \\ (A'_2)_{22}
                            \end{bmatrix}= B_1 \begin{bmatrix}
                                (A'_2)_{11}  \\ (A'_2)_{21}
                            \end{bmatrix} \\
                            (B_1^2+I)\begin{bmatrix}
                                (A'_2)_{11} \\ (A'_2)_{21}
                            \end{bmatrix} = \begin{bmatrix}
                                0 \\ 0
                            \end{bmatrix},
                        \end{array}
                    \right.
                \end{equation*}
                which implies, from $\det(B_1^2+I)\neq 0$, that $A'_2=\boldsymbol{0}$.
            \item By the same manner as previous item, if $\det(A_1-I)\neq 0$ then there exist $\alpha,\beta,\gamma,\delta\in \R$ so that
                \begin{equation*}
                    (A_1-I)\begin{bmatrix}
                        \alpha & \gamma\\
                        \beta & \delta
                    \end{bmatrix} = -A_2.
                \end{equation*}
                Equivalently, the matrix of $f$ with respect to the basis $\mathfrak{b}':=\{y_1,y_2,y_3+\alpha y_1+\beta y_2, y_4+\gamma y_1+\delta y_2\}$ is equal to
                $
                    \begin{bmatrix}
                        A_1 & \boldsymbol{0}\\
                        \boldsymbol{0} & I
                    \end{bmatrix}.
                $
                Once again, the commutation of $f$ and $g$ implies that the matrix of $g$ with respect to $\mathfrak{b}'$ is equal to
                $
                    \begin{bmatrix}
                        B_1 & \boldsymbol{0}\\
                        \boldsymbol{0} & J
                    \end{bmatrix}.
                $
                This completes the proof of the Lemma.
        \end{itemize}
    \end{proof}
    
    Now, we begin to prove Theorem \ref{classification-commutative3}. The proof falls into three parts. Firstly, we will prove that $\dim \Gc = 6$, and $\dim \Gc^1 \leq 4$. Secondly, we will prove that there is no $\MD{6}{4}$-algebra with $\dim\Gc^1=3$. Thirdly, we will classify $\MD{6}{4}$-algebras with $\dim\Gc^1=4$.  
    
    \begin{proof}[Proof of Theorem \ref{classification-commutative3}]
        Let's denote by $m$ the dimension of $\Gc^1$ ($m\geq 3$) and let $\mathfrak{b}$ be a basis of $\Gc$ which satisfies all conditions in Lemma \ref{cor2}. If so, 
        \begin{equation*}%\label{eq41}
            P_{x_n^*}=\begin{bmatrix}
                x_n^*([x_1,x_{n-m+1}]) & x_n^*([x_1,x_{n-m+2}]) & \cdots & x_n^*([x_1,x_n]) \\
                x_n^*([x_2,x_{n-m+1}]) & x_n^*([x_2,x_{n-m+2}]) & \cdots & x_n^*([x_2,x_n]) \\
                \vdots & \vdots &  & \vdots \\
                x_n^*([x_{n-m},x_{n-m+1}]) & x_n^*([x_{n-m},x_{n-m+2}]) & \cdots & x_n^*([x_{n-m},x_n]) \\
            \end{bmatrix}.
        \end{equation*}
        
        Because the matrices $[\adg{x_1}]_\mathfrak{b}, \ldots,[\adg{x_{n-m}}]_\mathfrak{b}$ are of block triangular form in the sense of Proposition \ref{triangular}, the first $(m-2)$ columns of $P_{x_n^*}$ are equal to zero. Hence, 
        \begin{equation*}
            \rank{(P_{x_n^*})} \leq 2.
        \end{equation*}
        By Remark \ref{remark210}, we obtain $\dim\Omega_{x_n^*}\leq 4$. Since each non-trivial coadjoint orbit of $\Gc$ is of dimension $n-2$, we get $n-2\leq 4$, i.e. $n\leq 6$. By the assumption, $n\geq 6$. Therefore, $n$ must be 6. In particular, $m = \dim \Gc^1 < \dim \Gc = 6$. 
        
        Now, we will prove that $m \leq 4$. Assume the contrary that $m=5$ then all but the first row of $P_{x_n^*}$ is zero. This turns out that $\dim\Omega_{x_n^*} \leq 2$, a contradiction to the fact that every non-trivial coadjoint orbit of an $\MD{n}{n-2}$-algebra is of dimension $n-2$. Hence, $3\leq m\leq 4$. 
        
        However, if $m=3$ then there is at least one block of size 1 in the triangular form of the matrices $\left\{[\adg{x_i}]_\mathfrak{b}: i=1,2,3\right\}$. In the other words, we may assume that 
        \begin{equation*}
            [\adg{x_1}]_\mathfrak{b}=\begin{bmatrix}
                    *_{2\times 2} & *\\
                    0 & a_1
                \end{bmatrix},
            [\adg{x_2}]_\mathfrak{b}=\begin{bmatrix}
                    *_{2\times 2} & *\\
                    0 & a_2
                \end{bmatrix},
            [\adg{x_3}]_\mathfrak{b}=\begin{bmatrix}
                    *_{2\times 2} & *\\
                    0 & a_3
                \end{bmatrix},
        \end{equation*}
        for some $a_1,a_2,a_3\in\R$. If so, 
        \begin{equation*}
            P_{x_6^*}=\begin{bmatrix}
                    0 & 0 & a_{1}\\
                    0 & 0 & a_{2}\\
                    0 & 0 & a_{3}
            \end{bmatrix}
        \end{equation*}
        which must have rank 1, or $\dim\Omega_{x_6^*}=2$, a contradiction. Therefore, $m=4$.
        
        Finally, let's classify $\MD{6}{4}$-algebras. By rewriting 
        \begin{equation*}
            [\adg{x_1}]_\mathfrak{b}=\begin{bmatrix}
                    A_1 & A_2\\
                    \boldsymbol{0} & A_3
                \end{bmatrix}, \text{ and }
            [\adg{x_2}]_\mathfrak{b}=\begin{bmatrix}
                    B_1 & B_2 \\
                    \boldsymbol{0} & B_3
                \end{bmatrix},
        \end{equation*}
        we have four possibilities for the $2\times 2$ matrices $A_3, B_3$ as follows:
        \begin{itemize}
            \item $A_3$ and $B_3$ are both of triangular form, i.e. $(A_3)_{21}=(B_3)_{21}=0$.
            \item $A_3=\lambda I_2$ and $B_3 = \begin{bmatrix}
                \mu & \zeta\\
                -\zeta & \mu
            \end{bmatrix}$ for some $\lambda, \mu\in\R, 0\neq \zeta \in \R$.
            \item $A_3=\begin{bmatrix}
                \mu & \zeta\\
                -\zeta & \mu
            \end{bmatrix}$ and $B_3 = \lambda I_2$ for some $\lambda,\mu\in\R, 0\neq \zeta \in \R$.
            \item $A_3=\begin{bmatrix}
                \lambda & \eta\\
                -\eta & \lambda
            \end{bmatrix}$ and $B_3=\begin{bmatrix}
                \mu & \zeta\\
                -\zeta & \mu
            \end{bmatrix}$ for some $\lambda,\eta,\mu,\zeta\in \R$ with $\eta \neq 0,\zeta \neq 0$.
        \end{itemize}
        
        Remark that the  change of basis $x_1\rightarrow x_1-\dfrac{\eta}{\zeta}x_2$ and the change of basis $x_1\leftrightarrow x_2$ bring respectively the fourth item and the third item to the second item. Hence, it is sufficient to consider only the two first possibilities. However, if $A_3$ and $B_3$ are both of triangular form, then 
        \begin{equation*}
            x_6^*([x_i,x_j])=0 \quad \forall 1\leq i,j\leq 5,
        \end{equation*}
        and hence, $\rank{(P_{x^*_6})}=1$, or $\dim\Omega_{x_6^*}=2$, a contradiction again. 
        
    Therefore, it suffices to consider the second item only: 
        \begin{equation*}
            A_3=\lambda I \text{ and } B_3 = \mu I + \zeta J \quad (\zeta \neq 0).
        \end{equation*}
    If so, by the same manner, we obviously obtain $\lambda \neq 0$. Now, by the following change of basis: 
        \begin{equation*}
            \left\{
                \begin{array}{ll}
                    x_1 & \rightarrow \frac{1}{\lambda} x_1\\
                    x_2 & \rightarrow\frac{1}{\zeta}(x_2-\mu x_1),
                \end{array}
            \right.
        \end{equation*}
        we may assume $\lambda=1,\mu=0$ and $\zeta=1$. 
        
    Hence, without loss of generality, we may assume from beginning that
        \begin{equation*}%\label{normal-form}
            [\adg{x_1}]_\mathfrak{b}=\begin{bmatrix}
                    A_1 & A_2\\
                    \boldsymbol{0} & I
                \end{bmatrix}, 
            [\adg{x_2}]_\mathfrak{b}=\begin{bmatrix}
                    B_1 & B_2 \\
                    \boldsymbol{0} & J
                \end{bmatrix}.
        \end{equation*}
    Similarly, we have two possibilities for the forms of $A_1$ and $B_1$ as follows:
        \begin{itemize}
            \item $A_1$ and $B_1$ are both of triangular form, i.e. $(A_1)_{21}=(B_1)_{21}=0$. 
            
            \item $A_1=\begin{bmatrix}
                \lambda & \eta\\
                -\eta & \lambda
            \end{bmatrix}$ and $B_1 = \begin{bmatrix}
                \mu & \zeta\\
                -\zeta & \mu
            \end{bmatrix}$ with $\eta^2+\zeta^2 \neq 0$. 
        \end{itemize}
    
    However, If $A_1$ and $B_1$ are both of triangular form then $\det(B_1^2+I)\neq 0$. It follows from Lemma \ref{lemma35} that we may assume $A_2=B_2=\boldsymbol{0}$. If so, it is elementary to check that  
        \begin{equation*}
            \left\{
                \begin{array}{ll}
                    x_4^*([x_1,x_5]) = x_4^*([x_1,x_6]) = x_4^*([x_2,x_5]) = x_4^*([x_2,x_6]) = 0,  \\
                    x_4^*([x_1,x_3])=x_4^*([x_2,x_3])=0.
                \end{array}
            \right.
        \end{equation*}
    Therefore, 
        \begin{equation*}
            P_{x_4^*}=\begin{bmatrix}
                0 & * & 0 & 0 \\
                0 & * & 0 & 0
            \end{bmatrix}
        \end{equation*}
        which has rank exactly 1. Hence, $\dim \Omega_{x_4^*}=2$, a contradiction again. 
        
    % Therefore, it suffices to assume that 
    %     \[
    %         A_1=\begin{bmatrix}
    %             \lambda & \eta\\
    %             -\eta & \lambda
    %         \end{bmatrix}, \text{ and } B_1 = \begin{bmatrix}
    %             \mu & \zeta\\
    %             -\zeta & \mu
    %         \end{bmatrix} \text{ with } (\eta,\zeta) \neq (0,0). 
    %     \] 
    %     %
    %     % 
    
    In summary, we may assume that 
        \begin{equation*}
            [\adg{x_1}]_{\mathfrak{b}}=\begin{bmatrix}
                    \lambda I+\eta J & A_2\\
                    \boldsymbol{0} & I
                \end{bmatrix}, 
            [\adg{x_2}]_{\mathfrak{b}}=\begin{bmatrix}
                    \mu I+\zeta J & B_2 \\
                    \boldsymbol{0} & J
                \end{bmatrix} \text{ with } \eta^2+\zeta^2\neq 0.
        \end{equation*}
        Besides, it is elementary to check that  
        \begin{equation*}
            \left\{
                \begin{array}{ll}
                    \det(\lambda I+\eta J-I) = 0   \Longleftrightarrow (\lambda,\eta)=(1,0)  \\
                    \det\left((\mu I+\zeta J)^2+I\right) = 0  \Longleftrightarrow (\mu,\zeta)= (0,\pm 1).
                \end{array}
            \right.
        \end{equation*}
        Hence, in light of Lemma \ref{lemma35}, we shall split the rest of the proof into two cases as followings:
        
        \begin{enumerate}
            \item \textbf{Case 1: $\boldsymbol{A_1=I}$ and $\boldsymbol{B_1=\pm J}$.} 
            If so, by the following change of basis: $x_4\rightarrow -x_4$ if necessary, we  can assume that $B_1=J$. In the other words,  
            \begin{equation*}
                [\adg{x_1}]_{\mathfrak{b}}=\begin{bmatrix}
                        I & A_2 \\
                        \boldsymbol{0} & I
                    \end{bmatrix}, 
                [\adg{x_2}]_{\mathfrak{b}}=\begin{bmatrix}
                        J & B_2\\
                        \boldsymbol{0} & J
                    \end{bmatrix}.
            \end{equation*}
            By the following change of basis: $x_5\rightarrow x_5+(B_2)_{12}x_3+(B_2)_{22}x_4$, we can assume $(B_2)_{12}=(B_2)_{22}=0$. If so, the commutation of $\adg{x_1}$ and $\adg{x_2}$ implies that 
            \begin{equation*}
                (A_2)_{11}=(A_2)_{22}, (A_2)_{12}=-(A_2)_{21}.
            \end{equation*}
            In the other words, we can assume that
            \begin{equation*}
                [\adg{x_1}]_{\mathfrak{b}}=\begin{bmatrix}
                        1 & 0 & \nu & \theta\\
                        0 & 1 & -\theta & \nu\\
                        0      & 0      & 1      & 0\\
                        0      & 0      & 0      & 1 
                    \end{bmatrix}, 
                [\adg{x_2}]_{\mathfrak{b}}=\begin{bmatrix}
                        0 & 1 & \chi & 0\\
                        -1 & 0 & \omega & 0\\
                        0      & 0      & 0      & 1\\
                        0      & 0      & -1     & 0
                    \end{bmatrix}.
            \end{equation*}
            Let's denote this Lie algebra by $L(\nu,\theta,\chi,\omega)$. Then, via the following change of basis:
            \begin{equation*}
                \left\{
                    \begin{array}{ll}
                        x_3\rightarrow (\chi+\omega)x_3-(\chi-\omega)x_4  \\
                        x_4\rightarrow (\chi -\omega) x_3+(\chi+\omega)x_4 \\
                        x_5\rightarrow x_5-x_6+\chi x_3+\omega x_4 \\
                        x_6\rightarrow x_5+x_6
                    \end{array}
                \right. \quad \text{ (if } \chi^2+\omega^2\neq 0),
            \end{equation*}
            we easily see that
            \begin{equation}\label{1}
                L(\nu,\theta,\chi,\omega) \cong L(\nu,\theta,1,0) \quad \text{ (if } \chi^2+\omega^2\neq 0).
            \end{equation}
        
        Remark that by basis changing: $x_3 \rightarrow -x_3$ if necessary, we can assume that $\nu \geq 0$.
            
        Similarly, via the following change of basis:
                \begin{equation*}
                    \left\{\begin{array}{ll}
                        x_3\rightarrow \nu x_3-\theta x_4  \\
                        x_4\rightarrow \theta x_3+\nu x_4,
                    \end{array}\right. \quad \text{ (if } \nu^2+\theta^2\neq 0),
                \end{equation*}
            we easily see that
            \begin{equation}\label{3}
                L(\nu,\theta,0,0) \cong L(1,0,0,0) \quad \text{ (if } \nu^2+\theta^2\neq 0).
            \end{equation}
            In summary. we conclude from the equations (\ref{1}) and (\ref{3}) that
            \begin{equation*}
                L(\nu,\theta,\chi,\omega) \cong \left\{\begin{array}{ll}
                    L(0,0,0,0) & \text{if } \nu^2+\theta^2=\chi^2+\omega^2=0 \\
                    L(1,0,0,0) & \text{if } \nu^2+\theta^2\neq 0, \text{ and } \chi^2+\omega^2= 0 \\
                    L(\nu,\theta,1,0) \quad (\text{with } \nu\geq 0) & \text{if } \chi^2+\omega^2 \neq 0 \\
                \end{array}\right.
            \end{equation*}
            Remark that $L(1,0,0,0)$ and $L(\nu,\theta,1,0)$ (with $\nu\geq 0$) are respectively isomorphic to $\mathfrak{s}_{6,211}$ and $\mathfrak{s}_{6,225}$ listed in \cite{Snob}. While $L(0,0,0,0)$ belongs to the family $\mathfrak{s}_{6,226}$ listed in \cite{Snob}. 
            
        \item \textbf{Case 2. Either $\boldsymbol{A_1\neq I}$ or $\boldsymbol{B_1\neq \pm J}$}. If so, we can assume that  $A_2=B_2=\boldsymbol{0}$ [Lemma \ref{lemma35}], or
            \begin{equation*}
                [\adg{x_1}]_\mathfrak{b}=\begin{bmatrix}
                        \lambda I + \eta J & \boldsymbol{0}\\
                        \boldsymbol{0} & I
                    \end{bmatrix},
                [\adg{x_2}]_\mathfrak{b}=\begin{bmatrix}
                        \mu I +\zeta J & \boldsymbol{0} \\
                        \boldsymbol{0} & J
                    \end{bmatrix} \text{ with } \eta^2+\zeta^2\neq 0.
            \end{equation*}
            Let's denote the corresponding Lie algebra as $L(\lambda,\eta,\mu,\zeta)$. 
            Then for any $F=a_1x_1^*+\cdots+a_6x_6^* \in\Gc^*$, we have
                \begin{equation*}
                    P_{F}=\begin{bmatrix}
                        \lambda a_3-\eta a_4 & \eta a_3+\lambda a_4 & a_5 & a_6 \\
                        \mu a_3 - \zeta a_4 & \zeta a_3+\mu a_4 & -a_6 & a_5
                    \end{bmatrix}.
                \end{equation*}
            Therefore, $\rank{(P_F)}=2$ for any $F\in\Gc^*$ with $F|_{\Gc^1}\neq 0$ if and only if $\lambda\zeta-\mu\eta \neq 0$. In the other words, $L(\lambda,\eta,\mu,\zeta)$ is an $\MD{6}{4}$-algebra if and only if 
            \begin{equation}\label{lambda}
                \lambda\zeta-\mu\eta \neq 0.
            \end{equation}
        Furthermore, by the following change of basis: 
        \begin{equation*}
            \left\{\begin{array}{ll}
                x_1 \rightarrow & \frac{1}{\lambda\zeta-\mu\eta}(\zeta x_1-\eta x_2)  \\
                x_2 \rightarrow & \frac{1}{\lambda\zeta-\mu\eta}(-\mu x_1+\lambda x_2) \\
                x_3 \leftrightarrow & x_5 \\
                x_4 \leftrightarrow & x_6,
            \end{array}\right.
        \end{equation*}
        we can see that
        \begin{equation}\label{condition}
            L(\lambda,\eta,\mu,\zeta)\cong L(\dfrac{\zeta}{\lambda\zeta-\mu\eta},-\dfrac{\eta}{\lambda\zeta-\mu\eta},-\dfrac{\mu}{\lambda\zeta-\mu\eta},\dfrac{\lambda}{\lambda\zeta-\mu\eta}).
        \end{equation} 
        
        Similarly, by the following change of basis: $x_4\rightarrow -x_4$, we get
            \begin{equation}\label{eq413}
                L(\lambda,\eta,\mu,\zeta) \cong L(\lambda,-\eta,\mu,-\zeta);
            \end{equation}
        and by the following change of basis:
                \begin{equation*}
                    \left\{
                        \begin{array}{lr}
                            x_2 \rightarrow &  -x_2\\
                            x_4 \rightarrow & - x_4\\
                            x_5  \rightarrow & x_6\\
                            x_6 \rightarrow  & x_5
                        \end{array}
                    \right.,
                \end{equation*}
        we get 
                \begin{equation}\label{eq411}
                    L(\lambda,\eta,\mu,\zeta) \cong L(\lambda,-\eta,-\mu,\zeta).
                \end{equation}
                
        \begin{itemize}
            \item If $\eta=0$ then it follows from the equation (\ref{lambda}) that $\lambda\zeta \neq 0$. Hence, the equation (\ref{condition}) becomes
            \begin{equation}\label{s226}
                L(\lambda,0,\mu,\zeta) \cong L(\dfrac{1}{\lambda},0,\dfrac{-\mu}{\lambda\zeta},\dfrac{1}{\zeta}).
            \end{equation}
            
            By combining the equations (\ref{eq413}), (\ref{eq411}) and (\ref{s226}), we obtain 
            \begin{equation*}
                L(\lambda,0,\mu,\zeta) \cong L(\lambda',0,\mu',\zeta') 
            \end{equation*}
            where $0<\zeta'\leq 1$, $\mu'\geq 0$, $\lambda' \neq 0$; and if $\zeta'=1$ then $|\lambda'|\leq 1$.
            This class of MD-algebras coincides with the family $\mathfrak{s}_{6,226}$ in \cite{Snob}, except some non MD-algebras cases. Hence, we also use the notation $\mathfrak{s}_{6,226}$ to denote this class. 
            
            \item If $\eta\neq 0$  then, by the same manner, we obtain
            \begin{equation*}
                L(\lambda,\eta,\mu,\zeta) \cong L(\lambda',\eta',\mu',\zeta') 
            \end{equation*}
            where $\lambda'\eta'-\mu'\zeta' >0$ and $\mu'\geq 0$. This class of MD-algebras coincides with the family $\mathfrak{s}_{6,228}$ in \cite{Snob}, except some non MD-algebras cases. Hence, we also denote this class by $\mathfrak{s}_{6,228}$. The proof is completed.
        \end{itemize}
        \end{enumerate}
    \end{proof}

\section[One step solvable MD{n-2}{n}-algebras-2]{One-step solvable $\boldsymbol{\MD{n}{n-2}}$-algebras which have low-dimensional derived algebras}\label{section-4}

    In order to obtain a complete classification of 1-step solvable $\MD{n}{n-2}$-algebras, we need to solve the problem for $\dim \Gc^1 \leq 2$. 
    The classification of Lie algebras which have low-dimensional derived algebras has been studied by T. Janisse \cite{Janisse10}, C. Sch\"obel \cite{Schobel93}, Vu A. L. et al. \cite{Vuthieu20}, F. Levstein \& A. L. Tiraboschi \cite{Levstein99}, and C.~Bartolone et al. \cite{Bartolone2011}. 
    % We first summary their classifications of solvable Lie algebras which have low-dimensional derived algebras and are not 2-step nilpotent.
    
    \begin{proposition}[\cite{Janisse10,Schobel93,Vuthieu20}]
        Let $\Gc$ be a real $n$-dimensional Lie algebra with $n\geq 5$.
        \begin{itemize}
            \item If $\dim\Gc^1\leq 2$ then $\Gc^1$ is commutative.
            
            \item If $\dim\Gc^1=1$ then $\Gc$ is an trivial extension of either $\textnormal{aff}(\R)$ or $\mathfrak{h}_{2m+1}$ ($n\geq 2m+1, m \geq 1$) %isomorphic to one of the following forms:
                %\begin{itemize}
                %    \item[(i)] $\textnormal{aff}(\R)\oplus \R^{n-2}$.
                %    \item[(ii)] %$\mathfrak{h}_{2m+1}\oplus \R^{n-2m-1}$ ($n\geq 2m+1, m \geq 1$).
                %\end{itemize}
            \item If $\dim\Gc^1=2$ and $\Gc^1$ is not completely contained in the centre $C(\Gc)$ of $\Gc$, then $\Gc$ is isomorphic to one of the following forms:
                \begin{itemize}
                    \item[(i)] $\Gc_{5+2k}:=\langle x_1, x_2, \dots,x_{5+2k}\rangle \ (n=5+2k,k\in \mathbb{N})$ with $[x_3,x_4]=x_1$ and
                        \begin{equation*}
                            [x_3,x_1]=[x_4,x_5]=\cdots=[x_{4+2k},x_{5+2k}]=x_2.
                        \end{equation*}
                    
                    \item[(ii)] $\Gc_{6+2k,1}:=\langle x_1,x_2,\dots,x_{6+2k}\rangle \ (n=6+2k,k\in \mathbb{N})$ with $[x_3,x_1]=x_1$ and 
                        \begin{equation*}
                            [x_3,x_4]=[x_5,x_6]=\cdots=[x_{5+2k},x_{6+2k}]=x_2.
                        \end{equation*}
                        
                    \item[(iii)] $\Gc_{6+2k,2}:=\langle x_1,x_2,\dots,x_{6+2k}\rangle \ (n=6+2k,k\in \mathbb{N})$ with $[x_3,x_4]=x_1$ and 
                        \begin{equation*}
                            [x_3,x_1]=[x_5,x_6]=\cdots=[x_{5+2k},x_{6+2k}]=x_2.
                        \end{equation*}
                    \item[(iv)] $\textnormal{aff}(\R)\oplus\mathfrak{h}_{2m+1}$ ($m \geq 1$).
                    \item[(v)] A trivial extension of one of Lie algebras listed above in (i), (ii), (iii) and (iv).
                    \item[(vi)] A trivial extension of $\textnormal{aff}(\R)\oplus\textnormal{aff}(\R)$.
                    \item[(vii)] A trivial extension of a Lie algebra $\Hc$ of dimension less than 5 such that $\dim\Hc^1=2$ and $\Hc^1$ is not contained in the centre of $\Hc$.
                \end{itemize}
        \end{itemize}
    \end{proposition}
    It is easy to see that $\Gc_{5+2k}, \Gc_{6+2k,1}$, $\Gc_{6+2k,2}$, $\textnormal{aff}(\R)\oplus \mathfrak{h}_{2m+1}$ and any trivial extension of $\textnormal{aff}(\R)\oplus \textnormal{aff}(\R)$ listed above are not MD-algebras for every $k$. For example, $\Gc_{5+2k}$ has a coadjoint orbit of dimension 2 and a coadjoint orbit of dimension $4+2k$:
    \begin{equation*}
        \dim \Omega_{x_1^*} = 2, \dim \Omega_{x_2^*} = 4+2k.
    \end{equation*}
    % Similarly, 
    % \begin{itemize}
    %     \item $\textnormal{aff}(\R)\oplus \textnormal{aff}(\R)$ is not an MD-algebra because it has a coadjoint orbit of dimension 2 and a coadjoint orbit of dimension 4.
    %     \item $\textnormal{aff}(\R)\oplus \mathfrak{h}_{2m+1}$ is not an MD-algebra because it has a coadjoint orbit of dimension 2 and a coadjoint orbit of dimension $2+2m$.
    % \end{itemize}
    % Remark that another reason so that the two later Lie algebras are not MD-algebras is that a nontrivial extension of any non-commutative Lie algebra is not an MD-algebra \cite[Theorem 3.1]{HHV21}. 
   \begin{corollary}\label{lemma-non2step}
       Let $\Gc$ be an $\MD{n}{n-2}$-algebra with $n\geq 6$. 
       \begin{itemize}
            \item If $\dim\Gc^1=1$ then $\Gc$ is isomorphic to $\mathfrak{h}_{2m+1}\oplus \R$ where $m=\frac{n-2}{2}$.
            \item If $\left\{\begin{array}{ll}
                \dim\Gc^1=2 \\
                \Gc^1\nsubseteq C(\Gc) 
            \end{array}\right.$
            then $\Gc$ is isomorphic to $\textnormal{aff}(\C)\oplus \R^2$.
        \end{itemize}
   \end{corollary}
   
%   Therefore, 
%     \begin{itemize}
%         \item If $\Gc$ is an $\MD{n}{n-2}$-algebra ($n\geq 5$) with $\dim\Gc^1=1$ then $\Gc$ is isomorphic to $\mathfrak{h}_{2m+1}\oplus \R$ where $m=\frac{n-2}{2}$.
%         \item If $\Gc$ is an $\MD{n}{n-2}$-algebra ($n\geq 5$) in which $\left\{\begin{array}{ll}
%             \dim\Gc^1=2 \\
%             \Gc^1\nsubseteq C(\Gc) 
%         \end{array}\right.$
%         then $\Gc$ is isomorphic to $\textnormal{aff}(\C)\oplus \R^2$.
%     \end{itemize}
    
    %These are Lie algebras defined in the first item of Theorem \ref{theorem62}. 
    Now, we will investigate the remaining case: 
    \begin{equation*}
        \left\{\begin{array}{ll}
                \dim\Gc^1=2 \\
                \Gc^1\subseteq C(\Gc) .
            \end{array}\right.
    \end{equation*}
    Firstly, it is easy to check that $\Gc^1\subseteq C(\Gc)$ if and only if $\Gc$ is 2-step nilpotent, i.e.  $\Gc_2:=[[\Gc,\Gc],\Gc]$ is trivial (a 2-step nilpotent Lie algebra is also called a metabelian Lie algebra).

    Because $\Gc$ is 2-step nilpotent  with $\dim\Gc^1=2$, there is a basis $\mathfrak{b}:=\{x_1,\dots,x_n\}$ of $\Gc$ such that $\Gc^1=\langle x_{n-1},x_n\rangle$ and $[x_i,x_{n-1}]=[x_i,x_n]=0$ for all $i$. Therefore, $\Gc$ determines a pair of $(n-2)\times (n-2)$ skew-symmetric matrices $(M,N)$ defined by 
    \begin{equation}\label{def-MN}
        (M)_{ij}: = x_{n-1}^*([x_i,x_j]); (N)_{ij}: = x_{n}^*([x_i,x_j]).
    \end{equation}
    Since $\dim\Gc^1=2$, $M$ and $N$ are linearly independent in the sense that there is no $(0,0) \neq (\alpha,\beta)$ such that $\alpha M+\beta N=\boldsymbol 0$. The matrices $(M,N)$ are called the {\em associated matrices} of $\Gc$ with respect to the basis $\mathfrak{b}$ (we also say that $\Gc$ is associated by the matrices $(M,N)$ with respect to $\mathfrak{b}$). Conversely, Let $(M,N)$ be any pair of skew-symmetric matrices of size $(n-2)\times (n-2)$ which are linearly independent. Then we can define a Lie algebra $\Gc$ of dimension $n$ as follows: $\Gc$ is spanned by a basis $\{x_1, \dots, x_n\}$, and the Lie brackets are defined via that basis as follows:
    \begin{equation*}
        \left\{\begin{array}{ll}
            [x_i,x_{n-1}] = [x_i,x_n] = 0 &  1 \leq i \leq n  \\
            {}[x_i,x_j] = (M)_{ij} x_{n-1}+(N)_{ij}x_n & 1 \leq i,j\leq n-2.
        \end{array}\right.
    \end{equation*}
    In 1999, F. Levstein \& A. L. Tiraboschi \cite{Levstein99} proved the corresponding between the isomorphism of two such 2-step nilpotent Lie algebras with the (strict) congruence of vector spaces spanned by their associated matrices, as stated in the following proposition.
    \begin{proposition}\cite{Levstein99}
        Let $\Gc$ and $\Gc'$ be two 2-step nilpotent Lie algebras which have $\dim\Gc^1=\dim\Gc'^1=2$. Suppose that $\Gc$ and $\Gc'$ are associated (with respect to some bases) with $(M,N)$ and $(M',N')$ respectively. Then $\Gc$ is isomorphic to $\Gc'$ if and only if there is a nonsingular matrix $T$ so that 
        \begin{equation*}
            T \cdot \langle M, N\rangle\cdot T^t = \langle M', N'\rangle
        \end{equation*}
    \end{proposition}
    
    In particular, if the pencils $M-\rho N$ and $M'-\rho N'$ are strictly congruent, i.e. there is a nonsingular matrix $T$ (which does not depend on $\rho$) so that $T(M-\rho N)T^t=M'-\rho N'$,  then their associated Lie algebras are isomorphic. Although the converse of the later statement is not true in general, but the statement is still useful to classify Lie algebras in this paper. The classification (up to strict congruence) of pencils of complex/real matrices which are either symmetric or skew-symmetric was solved by R. C. Thompson \cite{Thompson91} (the skew-symmetric case was classified in \cite{Scharlau76}). Because we are concerning with real skew-symmetric matrices, we will state his theorem for the case of pencils of real skew-symmetric matrices only. 
    \begin{proposition}\cite[Theorem 2]{Thompson91}\label{pencil}
        Let A and B be real skew-symmetric matrices. Then a simultaneous (real) congruence of $A$ and $B$ exists reducing $A-\rho B$ to a direct sum of types $m, \infty,\alpha$, and $\beta$, where 
        \begin{equation*}
            \begin{array}{cc}
                m:=\begin{bmatrix}
                    \boldsymbol 0 & L_{e}(\rho) \\
                    -L_e(\rho)^t & \boldsymbol 0
                \end{bmatrix}, 
                \infty: = \begin{bmatrix}
                    \boldsymbol 0 & \Delta_f -\rho \Lambda_f \\ -\Delta_f+\rho\Lambda_f & \boldsymbol 0
                \end{bmatrix}, \\
                \alpha: =\begin{bmatrix}
                    \boldsymbol 0 & (a -\rho) \Delta_g + \Lambda_g \\ (a+\rho)\Delta_g-\Lambda_g & \boldsymbol 0
                \end{bmatrix}, \beta:=\begin{bmatrix}
                    \boldsymbol 0 & \Gamma_h(\rho) \\
                    -\Gamma_h(\rho) & \boldsymbol 0 
                \end{bmatrix}
            \end{array}
        \end{equation*}
        with 
        \begin{equation*}
            L_e(\rho):=
                \begin{bmatrix}
                    \rho & \\
                    1 & \ddots & \\
                      & \ddots & -\rho \\
                      &        & 1
                \end{bmatrix}_{(e+1)\times e}, \Delta_f: = \begin{bmatrix}
                    & & 1 \\
                    & \iddots \\
                    1
                \end{bmatrix}_{f\times f},\Lambda_f: = \begin{bmatrix}
                    & & & 0\\
                    & & \iddots & 1\\
                    &\iddots & \iddots\\
                    0 & 1
                \end{bmatrix}_{f\times f}, 
        \end{equation*}
        and
        \begin{equation*}
            \Gamma_g(\rho) := \begin{bmatrix}
                    \boldsymbol 0 & \begin{bmatrix}
                        & & & R \\
                        & &\iddots& S \\
                        & \iddots & \iddots \\
                        R & S 
                    \end{bmatrix} \\
                    \begin{bmatrix}
                        & & & R \\
                        & &\iddots& S \\
                        & \iddots & \iddots \\
                        R & S 
                    \end{bmatrix} & \boldsymbol 0
                \end{bmatrix}_{g\times g}, R: = \begin{bmatrix}
                    c & d-\rho \\ d-\rho & -c
                \end{bmatrix}, S: = \begin{bmatrix}
                    0 & 1 \\ 1 & 0
                \end{bmatrix}
        \end{equation*}
        for some $a,c,d\in\mathbb R:c\neq 0$.
    \end{proposition}
    
    We can now return to the problem of classification of such 2-step nilpotent MD-algebras. According to Proposition \ref{rankbf}, $\dim\Omega_{F}=\rank \left(F([x_i,x_j])\right)_{n\times n}$ for every $0\neq F:=\lambda x_{n-1}^*+\mu x_n^* \in\Gc^*$. Hence, $\Gc$ is an $\MD{n}{k}$-algebra if and only if $\rank{(\lambda M+\mu N)} = k$ for every $(0,0)\neq (\lambda,\mu)\in\mathbb R^2$. 
    % In particular, when $k=n-2$, we obtain the following result.
    % \begin{lemma}
    %     Let $\Gc$ be a 2-step nilpotent Lie algebra of dimension $n$ such that $\dim\Gc^1=2$. Then $\Gc$ is an $\MD{n}{n-2}$-algebra if and only if every non-zero matrix in the vector space spanned by the associated matrices of $\Gc$ is nonsingular.
    % \end{lemma}
    Moreover, the type $\beta$ is the unique nonsingular type among the types $m,\infty,\alpha,\beta$ in the sense that every non-zero matrix of the type $\beta$ is nonsingular. This proves the following proposition.
    
    \begin{proposition}
        Let $\Gc$ be a 2-step nilpotent $\MD{n}{n-2}$-algebra such that $\dim\Gc^1=2$. Then there is a basis $\mathfrak{b}:=\{x_1, \dots, x_n\}$ of $\Gc$ so that $[x_i,x_{n-1}]=[x_i,x_n]=0$ for every $i$ and the associated pencil of $\Gc$ with respect to $\mathfrak{b}$ is equal to a direct sum of matrices of the form $\beta$ defined in Proposition \ref{pencil}.
    \end{proposition}
    \begin{corollary}
        If $\Gc$ is a 2-step nilpotent $\MD{n}{n-2}$-algebra which has $\dim\Gc^1=2$ then $n-2$ is divisible by $4$.
    \end{corollary}
    \begin{proof}
        It is straightforward from the fact that the type $\beta$ is of the size $(2g)\times (2g)$ where 2 divides $g$.
    \end{proof}
    % As a consequence, if $\Gc$ is a 2-step nilpotent $\MD{n}{n-2}$-algebra which has $\dim\Gc^1=2$ then $n-2$ is divisible by $4$. This completes the proof of Theorem \ref{theorem62}.

    Now, we will give illustrations for $n=6$ and $n=10$. 
    \begin{itemize}
        \item Let $n=6$. Then there is a basis $\{x_1,x_2,\dots,x_6\}$ of $\Gc_6$ such that $\Gc_6^1=\langle x_5,x_6\rangle$ and  
            \begin{equation*}
                \left(x_5^*([x_i,x_j])\right)_{4\times 4} = \begin{bmatrix}
                    0 & 0 & b & a \\
                    0 & 0 & a & -b\\
                    -b & -a & 0 & 0 \\
                    -a & b & 0 & 0
                \end{bmatrix}, \quad \left( x_6^*([x_i,x_j])\right)_{4\times 4} = \begin{bmatrix}
                    0 & 0 & 0 & -1 \\
                    0 & 0 & -1 & 0\\
                    0 & 1 & 0 & 0 \\
                    1 & 0 & 0 & 0
                \end{bmatrix}
            \end{equation*}
        for some non-zero $b\in\R$.
        By applying the change of basis: 
        \begin{equation*}
            \left\{
                \begin{array}{ll}
                    x_5  \rightarrow bx_5  \\
                    x_6\rightarrow ax_5-x_6,
                \end{array}
            \right.
        \end{equation*}
        we can assume $a=0$ and $b=1$. This Lie algebra is denoted as $\mathfrak{n}_{6,3}$ in \cite{Snob}.
    
        \item 
        Let $n=10$. Then there is a basis $\{x_1,x_2,\dots,x_{10}\}$ of $\Gc$ such that $\Gc^1=\langle x_9,x_{10}\rangle$ and  the associated pencil $M-\rho N:=\left(x_9^*([x_i,x_j])\right)_{8\times 8} - \rho\left(x_{10}^*([x_i,x_j])\right)_{8\times 8}$ is either a direct sum of two $4\times 4$ blocks of the type $\beta$ or just an $8\times 8$ matrix of the type $\beta$. Hence, we have either
            \begin{equation*}
                M -\rho N = \begin{bmatrix}
                    0 & 0 & b_1 & a_1 - \rho\\
                    0 & 0 & a_1 -\rho & -b_1\\
                    -b_1 & -a_1+\rho & 0 & 0 \\
                    -a_1+\rho & b_1 & 0 & 0 \\
                    &&&& 0 & 0 & b_2 & a_2 - \rho \\
                    &&&& 0 & 0 & a_2 -\rho & -b_2\\
                    &&&&-b_2 & -a_2-\rho & 0 & 0 \\
                    &&&& -a_2-\rho & b_2 & 0 & 0
                \end{bmatrix}
            \end{equation*}
            or 
            \begin{equation*}
                M -\rho N = \begin{bmatrix}
                    0 & 0 & 0 & 0 & 0 & 0 & b_1 & a_1 - \rho\\
                    0 & 0 & 0 & 0 & 0 & 0 & a_1 -\rho & -b_1\\
                    0 & 0 & 0 & 0 & b_1 & a_1-\rho & 0 & 1 \\
                    0 & 0 & 0 & 0 & a_1-\rho & -b_1 & 1 & 0 \\
                    0 & 0 & -b_1 & -a_1 - \rho & 0 & 0 & 0 & 0\\
                    0 & 0 & -a_1 -\rho & b_1 & 0 & 0 & 0 & 0\\
                    -b_1 & -a_1-\rho & 0 & -1 & 0 & 0 & 0 & 0\\
                    -a_1-\rho & b_1 & -1 & 0 & 0 & 0 & 0 & 0\\
                \end{bmatrix}
            \end{equation*}
            for some non-zero $b_1,b_2\in\R$. Equivalently, $\Gc$ is isomorphic to one of the following forms:
        \begin{itemize}
            \item[(i)] $\Gc_{10,1}(a_1,b_1,a_2,b_2) := \langle x_1, x_2, \dots,x_{10}\rangle $ with $[x_i,x_9]=[x_i,x_{10}]=0$ for all $i$ and 
            \begin{center}
                \begin{tabular}{|c|c|c|c|c|c|c|c|}
                    \hline
                    & $x_2$ & $x_3$ & $x_4$ & $x_5$ & $x_6$ & $x_7$ & $x_8$\\
                    \hline
                    $x_1$ & 0 & $b_1x_9$ & $a_1x_9-x_{10}$ & 0 & 0 & 0 & 0\\
                    \hline
                    $x_2$ &   & $a_1x_9-x_{10}$ & $-b_1x_9$ & 0 & 0 & 0 & 0\\
                    \hline
                    $x_3$ &   &       & 0     & 0 & 0 & 0 & 0\\
                    \hline
                    $x_4$ &   &       &       & 0 & 0 & 0 & 0\\
                    \hline
                    $x_5$ &   &       &       & & 0 & $b_2x_9$ & $a_2x_9-x_{10}$\\
                    \hline
                    $x_6$ &   &       &       & &  & $a_2x_9-x_{10}$ & $-b_2x_9$\\
                    \hline
                    $x_7$ &   &       &       &  &  &  & 0\\
                    \hline
                \end{tabular}
                \\ $(b_1b_2\neq 0)$
            \end{center}
            
            If so, by the change of basis: 
            \begin{equation*} 
                x_i\leftrightarrow x_{i+4}: i \in \{1,2,3,4\},
            \end{equation*} 
            we easily see that 
            \begin{equation}\label{first}
                \Gc_{10,1}(a_1,b_1,a_2,b_2) \cong \Gc_{10,1}(a_2,b_2,a_1,b_1).
            \end{equation}
            Similarly, by the following change of basis:
            \begin{equation*}
                \left\{\begin{array}{ll}
                    x_{10} & \rightarrow -a_1x_9+x_{10}\\
                    x_9 & \rightarrow b_1x_9,
                \end{array}\right.
            \end{equation*}
            we obtain 
            \begin{equation}\label{second}
                \Gc_{10,1}(a_1,b_1,a_2,b_2) \cong \Gc_{10,1}(0,1,\dfrac{a_2-a_1}{b_1},\dfrac{b_2}{b_1}).
            \end{equation}
            We conclude from the isomorphism (\ref{first}) that we always can assume $0<|b_2|\leq |b_1|$, and from the isomorphism (\ref{second}) that $a_1=0, b_1=1$, i.e., 
            \begin{equation*}
                \Gc_{10,1}(a_1,b_1,a_2,b_2) \cong \Gc_{10,1}(0,1,\mu,\lambda) \quad (0<|\lambda|\leq 1)
            \end{equation*}
            % This is the family $\Gc_{10,1}(\lambda,\mu)$ defined in Theorem \ref{pro52}.
            
            \item[(ii)] $\Gc_{10,2}(a_1,b_1):= \langle x_1, x_2, \dots,x_{10}\rangle $ with $[x_i,x_9]=[x_i,x_{10}]=0$ for all $i$ and 
            \begin{center}
                \begin{tabular}{|c|c|c|c|c|c|c|c|}
                    \hline
                    & $x_2$ & $x_3$ & $x_4$ & $x_5$ & $x_6$ & $x_7$ & $x_8$\\
                    \hline
                    $x_1$ & 0 & 0 & 0 & 0 & 0 & $b_1x_9$ & $a_1x_9-x_{10}$\\
                    \hline
                    $x_2$ &   & 0 & 0 & 0 & 0 & $a_1x_9-x_{10}$ & $-b_1x_9$\\
                    \hline
                    $x_3$ &   &       & 0   & $b_1x_9$ & $a_1x_9-x_{10}$ & 0 & $x_9$\\
                    \hline
                    $x_4$ &   &       &       & $a_1x_9-x_{10}$ & $-b_1x_9$ & $x_9$ & 0\\
                    \hline
                    $x_5$ &   &       &       & & 0 & 0 & 0\\
                    \hline
                    $x_6$ &   &       &       & &  & 0 & 0\\
                    \hline
                    $x_7$ &   &       &       &  &  &  & 0\\
                    \hline
                \end{tabular}
                \\ $(b_1\neq 0)$
            \end{center}
        \end{itemize}
        By the change of basis: $x_{10} \rightarrow a_1x_9-x_{10}$, we easily see that
        \begin{equation*}
            \Gc_{10,2}(a_1,b_1) \cong \Gc_{10,2}(0,\lambda) \quad (\lambda \neq 0).
        \end{equation*}
    \end{itemize}
    
    In summary, we have proven the following theorem.
    
    \begin{theorem}\label{theorem62}
        Let $\Gc$ be an $\MD{n}{n-2}$-algebra of dimension $n\geq 6$ with $\dim\Gc^1\leq 2$. 
        \begin{enumerate}
            \item If $\Gc$ is not 2-step nilpotent, i.e. $[\Gc,\Gc^1]\neq 0$, then $\Gc$ is either isomorphic to $\R^2\oplus\textnormal{aff}(\C)$ or isomorphic to 
            $\R\oplus \mathfrak{h}_{2m+1}$ where $2m=n-2$.
            \item If $\Gc$ is a 2-step nilpotent Lie algebra then $n=4k+2$ for some $k\in\mathbb N$, and the associated pencil of $\Gc$ is a direct sum of type $\beta$.
            \item If $n=6$ then $\Gc$ is isomorphic to $\mathfrak{n}_{6,1}$ defined in \cite{Snob}.
            \item If $n=10$ then $\Gc$ is isomorphic to one of the following families: 
            $\Gc_{10,1}(0,1,\mu,\lambda) \ (0<|\lambda|\leq 1)$ and $\Gc_{10,2}(0,\lambda) \ (\lambda \neq 0)$.
        \end{enumerate}
    \end{theorem}

\section[Two-step solvable MD{n}{n-2}-algebras]{Two-step solvable $\boldsymbol{\MD{n}{n-2}}$-algebras}\label{section-5}
    Finally, to complete the classification of $\MD{n}{n-2}$-algebras, we only need to classify 2-step solvable $\MD{n}{n-2}$-algebras. Surprising, such a Lie algebra is decomposable and has dimension exactly 6. 
    \begin{theorem}\label{main-noncommutative}
        Let $\Gc$ be a 2-step solvable real Lie-algebra whose non-trivial coadjoint orbits are all of codimension 2. Then $\Gc$ is isomorphic to $\R\oplus \mathfrak{s}_{5,45}$. 
    \end{theorem}
    
    \begin{proof}
    Recall that for every $x,y,z\in\Gc$, we have:
    \begin{equation*}
        \left[[x,y],z\right] = \left[x,[y,z]\right]-\left[y,[x,z]\right].
    \end{equation*}
    It follows that 
    \begin{equation*}
        \ad{x}\ad{y}-\ad{y}\ad{x}=\ad{[x,y]}.
    \end{equation*}
    Hence, for every $x\in \Gc^1$, we have 
        \begin{equation}\label{commutative}
            \text{trace}(\ad{x})=\text{trace}(\adg{x})=\text{trace}(\adgg{x})=0 .
        \end{equation}
    According to Theorem 3.5 in \cite{HHV21}, $1\leq \dim\Gc^2\leq 2$. Therefore, we will divide the proof into two cases:
    \begin{itemize}
        \item \textbf{Case 1: $\boldsymbol{\dim\Gc^2=2}$.} If so, $\Hc:=\Gc/\Gc^2$ is a 1-step solvable Lie algebra whose non-trivial coadjoint orbits are all of the same dimension as $\Hc$ \cite[Theorem 3.5]{HHV21}. In the other words, $\mathcal{H}$ is an $\MD{n}{n}$-algebra. According to Proposition \ref{sonviet}, $\Hc$ is isomorphic to either $\textnormal{aff}(\R)$ or $\textnormal{aff}(\C)$. Since $\dim\Gc\geq 6$, $\Hc \cong \textnormal{aff}(\C)$. It implies the existence of a basis $\mathfrak{b}:=\{x_1,x_2,y_1,y_2,z_1,z_2\}$ of $\Gc$ such that: 
        \begin{align*}
            \Gc^1= & \langle y_1,y_2,z_1,z_2\rangle, \quad \Gc^2=  \langle z_1,z_2 \rangle\\
            \Hc = & \langle \overline{x_1}, \overline{x_2},\overline{y_1},\overline{y_2}\rangle \cong \text{aff}(\C),
        \end{align*}
        where 
        \begin{align*}
            [\overline{x_1},\overline{y_1}]=\overline{y_1}, [\overline{x_1},\overline{y_2}]=\overline{y_2} \text{ and } [\overline{x_2},\overline{y_1}]=\overline{y_2}, [\overline{x_2},\overline{y_2}]=-\overline{y_1}.
        \end{align*}
        Since $\Gc^1$ and $\Gc^2$ are both ideals of $\Gc$, the Lie brackets in $\Gc$ can be determined as follows:
        \begin{equation*}
            \begin{array}{|c|c|c|c|c|c|c|}
                \hline
                     & x_1 & x_2 & y_1 & y_2 & z_1 & z_2\\
                \hline
                x_1  & 0 & \lambda_1z_1+\lambda_2z_2 & 
                y_1+\lambda_3z_1+\lambda_4z_2 & y_2+\lambda_5z_1+\lambda_6z_2 & \lambda_7 z_1+\lambda_8z_2 & \lambda_9z_1+\lambda_{10}z_2\\ 
                \hline
                x_2 &  & 0 & y_2+\lambda_{11}z_1+\lambda_{12}z_2 & -y_1+\lambda_{13}z_1+\lambda_{14}z_2 & \lambda_{15} z_1+\lambda_{16}z_2 & \lambda_{17}z_1+\lambda_{18}z_2\\
                \hline
                y_1 &  &  & 0 & \lambda_{19}z_1+\lambda_{20}z_2 & \lambda_{21} z_1+\lambda_{22}z_2 & \lambda_{23}z_1+\lambda_{24}z_2\\
                \hline
                y_2 & &  &  & 0 & \lambda_{25} z_1+\lambda_{26}z_2 & \lambda_{27}z_1+\lambda_{28}z_2\\
                \hline
                z_1 &  &  &  &  & 0  & 0\\
                \hline
            \end{array}
        \end{equation*}
    Since $\Gc^2$ is commutative, we can obtain directly from the Jacobi identity that $\adgg{y_1}\adgg{y_2}=\adgg{y_2}\adgg{y_1}$. By Proposition \ref{triangular}, we can assume that $[\adgg{y_1}]_{\mathfrak{b}}$ and $[\adgg{y_2}]_{\mathfrak{b}}$ are both either of the diagonal form or of the form $aI+bJ$. 
    Without loss of generality, we can assume that
    \begin{equation*}
        \text{either} \quad 
        \left\{
            \begin{array}{ll}
                [\adgg{y_1}]_\mathfrak{b} = 
                \begin{bmatrix}
                    a & 0 \\ 0 & b
                \end{bmatrix} \\
                {}[\adgg{y_2}]_\mathfrak{b} = 
                \begin{bmatrix}
                    c & 0 \\ 0 & d
                \end{bmatrix}
            \end{array}
        \right. \quad \text{or} \quad 
        \left\{
            \begin{array}{ll}
                [\adgg{y_1}]_\mathfrak{b} = \begin{bmatrix}
                    a & b \\ -b & a
                \end{bmatrix} \\
                {}[\adgg{y_2}]_\mathfrak{b} = \begin{bmatrix}
                    c & d \\ -d & c
                \end{bmatrix}.
            \end{array}
        \right.
    \end{equation*}
    
    Moreover, it follows from the equation (\ref{commutative}) that
    \begin{equation*}
        \textnormal{trace}(\adgg{y_1})=\textnormal{trace}(\adgg{y_2}) =0.
    \end{equation*}
    It turns out that 
    \begin{equation*}
        \text{either} \quad 
        \left\{
            \begin{array}{ll}
                [\adgg{y_1}]_\mathfrak{b} = 
                \begin{bmatrix}
                    a & 0 \\ 0 & -a
                \end{bmatrix} \\
                {}[\adgg{y_2}]_\mathfrak{b} = 
                \begin{bmatrix}
                    c & 0 \\ 0 & -c
                \end{bmatrix}
            \end{array}
        \right. \quad \text{or} \quad 
        \left\{
            \begin{array}{ll}
                [\adgg{y_1}]_\mathfrak{b} = \begin{bmatrix}
                    0 & b \\ -b & 0
                \end{bmatrix} \\
                {}[\adgg{y_2}]_\mathfrak{b} = \begin{bmatrix}
                    0 & d \\ -d & 0
                \end{bmatrix}.
            \end{array}
        \right.
    \end{equation*}
    In both cases, there is $(0,0) \neq (\lambda,\mu)\in\mathbb{R}^2$ so that $\lambda \adgg{y_1}+\mu\adgg{y_2}=\boldsymbol 0$. Now, by applying the Jacobi identity  to $(x_2,\lambda y_1+\mu y_2,z)$ for any $z\in\Gc^2$, we easily see that
    \begin{equation*}
        \boldsymbol 0=\adgg{x_2}\adgg{\lambda y_1+\mu y_2}-\adgg{\lambda y_1+\mu y_2}\adgg{x_2}=\adgg{[x_2,\lambda y_1+\mu y_2]}= -\mu\adgg{y_1}+\lambda \adgg{y_2}.
    \end{equation*}
    Therefore, 
    \begin{equation*}
            \lambda \adgg{y_1}+\mu\adgg{y_2} =
             -\mu\adgg{y_1}+\lambda \adgg{y_2}=\boldsymbol 0.
    \end{equation*}
    This clearly forces $\adgg{y_1}=\adgg{y_2}=\boldsymbol 0$, and consequently $\Gc^2$ is spanned by $\{[y_1,y_2]\}$, a contradiction to $\dim\Gc^2=2$. Hence, this case is excluded.

        \item \textbf{Case 2: $\boldsymbol{\dim\Gc^2=1}$.} If so, $\Hc:=\Gc/\Gc^2$ is a 1-step solvable Lie-algebra whose non-zero coadjoint orbits are of codimension 1. It follows from Proposition \ref{pro210} that $\Hc$ is isomorphic to one of the followings: $\mathfrak{h}_{2m+1}$, $\R\oplus \textnormal{aff}(\C)$. Furthermore, if $\mathcal{H}\cong \mathfrak{h}_{2m+1}$ then $\dim\Gc^1=2$ and $\dim\Gc^2=1$. This is impossible because $\Gc^1$ is nilpotent.
        %according to Theorem \ref{theorem62}, there is no 2-step solvable $\MD{n}{n-2}$-algebra $\Gc$ which has $\dim\Gc^1=2$ (Actually, for any Lie algebra $\Gc$, . 
        Hence, $\mathcal{H}\cong \R\oplus \textnormal{aff}(\C)$. 
        
        Equivalently, we can fix a basis $\{x_1,x_2,y_1,y_2,y_3,z\}$ of $\Gc$ so that 
        \begin{equation*}
            \left\{
                \begin{array}{ll}
                    \Gc^1= & \langle y_1,y_2,y_3\rangle, \quad \Gc^2=  \langle z \rangle, \\
                    \Hc = & \langle\overline{y_3}\rangle\oplus\langle \overline{x_1}, \overline{x_2},\overline{y_1},\overline{y_2}\rangle 
                        \end{array}
            \right.
        \end{equation*}
        where the Lie brackets in $\Hc$ are the same as those in $\R\oplus \textnormal{aff}(\C)$, i.e. 
        \begin{align*}
            [\overline{x_1},\overline{y_1}]=\overline{y_1}, [\overline{x_1},\overline{y_2}]=\overline{y_2} \text{ and } [\overline{x_2},\overline{y_1}]=\overline{y_2}, [\overline{x_2},\overline{y_2}]=-\overline{y_1}.
        \end{align*}
        It implies that the Lie brackets in $\Gc$ must have the form
        \[\begin{array}{|c|c|c|c|c|c|c|c|}
                \hline
                    & x_1 & x_2 & y_1 & y_2 & y_3 & z\\
                \hline
                x_1 &  & \lambda_1z & y_1+\lambda_2z & y_2+\lambda_3z & \lambda_4 z & \lambda_5z\\ 
                \hline
                x_2 &  &  &  y_2+\lambda_6z & -y_1+\lambda_7z & \lambda_8 z & \lambda_9z\\
                \hline
                y_1 & &  &  & \lambda_{10}z & \lambda_{11} z & \lambda_{12}z\\
                \hline
                y_2 & &  &  &  & \lambda_{13} z & \lambda_{14}z\\
                \hline
                y_3   & &  &  &  &  & \lambda_{15}z\\
                \hline
        \end{array}\]
        
        If so, it follows from the equation (\ref{commutative}) that  
        \begin{equation*}
            \lambda_{12}=\lambda_{14}=0.
        \end{equation*}
        This means $[y_1,z]=[y_2,z]=0$. Because $\Gc^2\neq \{0\}$, we must have $\lambda_{10}\neq 0$. By basis changing $z\rightarrow \dfrac{1}{\lambda_{10}}z$, we may assume $\lambda_{10}=1$. 
        
        Now, by checking the Jacobi identity to the following triples
        $(x_1,y_1,y_2)$; $(x_2,y_1,y_2)$; $(y_1,y_2,y_3)$; $(x_1,x_2,y_3)$; $(x_1,y_1,y_3)$; and $(x_1,y_2,y_3)$; we obtain
        \begin{equation*}
            \lambda_{5}=2, \lambda_{9}=\lambda_{15}=2\lambda_8+\lambda_1\lambda_{15}=\lambda_{11}+\lambda_2\lambda_{15}=\lambda_{13}+\lambda_3\lambda_{15}=0
        \end{equation*}
        
        Hence, 
        \begin{equation*}
            \lambda_5=2, \lambda_{8}=\lambda_9=\lambda_{11}=\lambda_{12}=\lambda_{13}=\lambda_{14}=\lambda_{15}=0.
        \end{equation*}
        By basis changing $y_3\rightarrow y_3-\dfrac{\lambda_4}{2}z$ if necessary, we get $\Gc$ decomposable. In the other words, $\Gc$ is isomorphic to a direct sum of $\R$ with a Lie algebra $\Gc'$. Since $\Gc$ is 2-step solvable, so is $\Gc'$. Furthermore, non-zero coadjoint orbits of $\Gc'$ and $\Gc$ have the same dimension \cite[Theorem 3.1]{HHV21}. In the other words, $\Gc'$ is a 2-step solvable MD-algebra whose non-trivial coadjoint orbits are all of codimension 1. According to Proposition \ref{pro210}, $\Gc'$ must be isomorphic to $\mathfrak{s}_{5,45}$. Equivalently, $\Gc$ is isomorphic to $\mathbb{R}\oplus\mathfrak{s}_{5,45}$. This completes the proof.  
    \end{itemize}
\end{proof}

\section{Concluding Remarks}
In summary, the paper has introduced the  classification of $\MD{n}{n-2}$-class with $2 \leq n \in \mathbb{N}$. There are 14 different $\MD{n}{n-2}$-algebras (up to an isomorphism) of dimension $n < 5$ listed in Table \ref{table:1}. The subclass of all 2-step nilpotent $\MD{n}{n-2}$-algebras with $n \geq 6$ is classified by canonical forms of associated pencils of matrices,
%a pair of the associated matrices of pencils (up to a strict congruence), 
in which algebras of dimension $n \leq 10$ are listed in Table \ref{table:3}. The remaining subclass of $\MD{n}{n-2}$-algebras is classified (up to an isomorphism) and listed in Table \ref{table:2}. In the following tables, $\{x_1,x_2, \dots, x_n \}$ is used to denote a basis of corresponding $\MD{n}{n-2}$-algebra $\Gc$.

\begin{table}[!ht]
\centering
{\small
\caption{List of all $\MD{n}{n-2}$-algebras with $n=2,4$.}
\label{table:1}
 \begin{tabular}{l l l l} 
 \hline
 $n$ & Algebras & Non-trivial Lie brackets & Notes \\ 
 \hline
 2 & $\R^2$ & - & \\
 \hline
 4 & $\mathfrak{n}_{4,1}$ & $[x_2,x_4]=x_1, [x_3,x_4]=x_2$ & \\
 & $\mathfrak{s}_{4,1}$ & $[x_4,x_2]=x_1, [x_4,x_3]=x_3$ & \\
 & $\mathfrak{s}_{4,2}$ & $[x_4,x_1]=x_1, [x_4,x_2]=x_1+x_2, [x_4,x_3]=x_2+x_3$ & \\
 & $\mathfrak{s}_{4,3}$ & $[x_4,x_1]=x_1, [x_4,x_2]=\alpha x_2, [x_4,x_3]=\beta x_3$ & $0<|\beta|\leq |\alpha| \leq 1$, $(\alpha,\beta) \neq (-1,-1)$ \\
 & $\mathfrak{s}_{4,4}$ & $[x_4,x_1]=x_1, [x_4,x_2]=x_1+x_2, [x_4,x_3]=\alpha x_3$ & $\alpha \neq 0$ \\
 & $\mathfrak{s}_{4,5}$ & $[x_4,x_1]=\alpha x_1, [x_4,x_2]=\beta x_2-x_3, [x_4,x_3]=x_2+\beta x_3$ & $\alpha > 0$ \\
 & $\mathfrak{s}_{4,6}$ & $[x_4,x_2]=x_2, [x_4,x_3]=-x_3$ & \\
 & $\mathfrak{s}_{4,7}$ & $[x_4,x_2]=-x_3, [x_4,x_3]=x_2$ & \\
 & $\textnormal{aff}(\R) \oplus \R^2$ &$[x_1,x_2]=x_2$ & \\
 & $\mathfrak{n}_{3,1} \oplus \R$ &$[x_2,x_3]=x_1$ & \\
 & $\mathfrak{s}_{3,1} \oplus \R$ &$[x_3,x_1]=x_1, [x_3,x_2]=\alpha x_2$  & $0< |\alpha| \leq 1$\\
 & $\mathfrak{s}_{3,2} \oplus \R$ &$[x_3,x_1]=x_1, [x_3,x_2]=x_1+x_2$  & \\
 & $\mathfrak{s}_{3,3} \oplus \R$ &$[x_3,x_1]=\alpha x_1-x_2, [x_3,x_2]=x_1+\alpha x_2$  & $\alpha \geq 0$ \\
 \hline
 \end{tabular}  }
\end{table}

 \begin{table}[!ht]
\centering
{\small
\caption{List of all $\MD{n}{n-2}$-algebras with $n \geq 6$ which are not 2-step nilpotent.}
\label{table:2}
 \begin{tabular}{l l l l} 
 \hline
 $\dim \Gc^1$ & Algebras & Non-trivial Lie brackets & Notes \\ 
 \hline
  1& $ \mathfrak{h}_{2m+1} \oplus \R$ & $[x_i,x_{m+i}]=x_{2m+1}$ $\forall i=1,\dots,m$ & $2m=n-2$ \\
 \hline
  2& $ \textnormal{aff}(\C) \oplus \R^2$ & $[x_3,x_1]=-x_2, [x_3,x_2]=[x_4,x_1]=x_1, [x_4,x_2]=x_2$ & \\ \hline
 $\geq 3$ & $\mathfrak{s}_{6,211}$ & 
                            \begin{tabular}{|c|c|c|c|c|}
                            \hline
                                 $[\cdot,\cdot]$ & $x_3$ & $x_4$ & $x_5$ & $x_6$\\
                            \hline
                            $x_1$ & $x_3$ & $x_4$ & $x_5+x_3$ & $x_6+x_4$\\
                            \hline
                            $x_2$ & $-x_4$ & $x_3$ & $-x_6$ & $x_5$ \\
                            \hline
                            \end{tabular}
                            & \\
 & $\mathfrak{s}_{6,225}(\nu,\theta)$ &
                            \begin{tabular}{|c|c|c|c|c|}
                            \hline
                                 $[\cdot,\cdot]$ & $x_3$ & $x_4$ & $x_5$ & $x_6$\\
                            \hline
                            $x_1$ & $x_3$ & $x_4$ & $x_5+\nu x_3-\theta x_4$ & $x_6+\theta x_3+\nu x_4$\\
                            \hline
                            $x_2$ & $-x_4$ & $x_3$ & $-x_6+x_3$ & $x_5$ \\
                            \hline
                            \end{tabular}
                            & $\nu\geq 0$ \\
 & $\mathfrak{s}_{6,226}(\lambda,\mu,\zeta)$ & 
                            \begin{tabular}{|c|c|c|c|c|}
                            \hline
                                 $[\cdot,\cdot]$ & $x_3$ & $x_4$ & $x_5$ & $x_6$\\
                            \hline
                            $x_1$ & $\lambda x_3$ & $\lambda x_4$ & $x_5$ & $x_6$\\
                            \hline
                            $x_2$ & $\mu x_3-\zeta x_4$ & $\zeta x_3+\mu x_4$ & $-x_6$ & $x_5$\\
                            \hline
                            \end{tabular} 
                            & $\begin{cases}
                                \lambda \neq 0,\mu\geq 0, 0<\zeta \leq 1 \\
                                \mbox{if } \zeta=1 \mbox{ then } |\lambda| \leq 1
                            \end{cases} $
                            \\
 & $\mathfrak{s}_{6,228}(\lambda,\mu,\eta,\zeta)$ & 
                            \begin{tabular}{|c|c|c|c|c|}
                            \hline
                                 $[\cdot,\cdot]$ & $x_3$ & $x_4$ & $x_5$ & $x_6$\\
                            \hline
                            $x_1$ & $\lambda x_3-\eta x_4$ & $\eta x_3+\lambda x_4$ & $x_5$ & $x_6$\\
                            \hline
                            $x_2$ & $\mu x_3-\zeta x_4$ & $\zeta x_3+\mu x_4$ & $-x_6$ & $x_5$ \\
                            \hline
                            \end{tabular} & $\lambda\zeta-\mu\eta>0, \mu\geq 0$
                            \\
 & $\mathfrak{s}_{5,45} \oplus \R$ & 
                            \begin{tabular}{|c|c|c|c|}
                            \hline
                                 $[\cdot,\cdot]$ & $x_1$ & $x_2$ & $x_3$\\
                            \hline
                            $x_2$ & 0 & 0 & $x_1$\\
                            \hline
                            $x_4$ & $2x_1$ & $x_2$ & $x_3$\\
                            \hline
                            $x_5$ & 0 & $x_3$ & $-x_2$ \\
                            \hline
                            \end{tabular} 
                            & \\ \hline
 \end{tabular}  }
\end{table}

 \begin{table}[!ht]
\centering
{\small
\caption{List of all 2-step nilpotent $\MD{n}{n-2}$-algebras with $6 \leq n \leq 10$.}
\label{table:3}
 \begin{tabular}{l l l l} 
 \hline
 $n$ & Algebras & Non-trivial Lie brackets & Notes \\ 
 \hline
 6 & $\mathfrak{n}_{6,1}$ & $[x_4,x_5]=x_2, [x_4,x_6]=x_3, [x_5,x_6]=x_1$ & \\
 \hline
 8 & There is no $\MD{8}{6}$-algebra \\
 \hline
 10 & $\Gc_{10,1}(0,1,\mu,\lambda)$ & 
                            \begin{tabular}{|c|c|c|c|c|}
                            \hline
                                 $[\cdot,\cdot]$ & $x_3$ & $x_4$ & $x_7$ & $x_8$\\
                            \hline
                            $x_1$ & $x_9$ & $-x_{10}$ & 0 & 0\\
                            \hline
                            $x_2$ & $-x_{10}$ & $-x_9$ & 0 & 0 \\
                            \hline
                             $x_5$ & 0 & 0 & $\lambda x_9$ & $\mu x_9-x_{10}$ \\
                            \hline
                            $x_6$ & 0 & 0 & $\mu x_9-x_{10}$ & $-\lambda x_9$ \\
                            \hline
                            \end{tabular}
                            & $0<|\lambda|\leq 1$\\
 & $\Gc_{10,2}(0,\lambda)$  & 
                            \begin{tabular}{|c|c|c|c|c|}
                            \hline
                                 $[\cdot,\cdot]$ & $x_5$ & $x_6$ & $x_7$ & $x_8$\\
                            \hline
                            $x_1$ & 0 & 0 & $\lambda x_9$ & $-x_{10}$\\
                            \hline
                            $x_2$ & 0 & 0 & $-x_{10}$ & $-\lambda x_9$ \\
                            \hline
                             $x_3$ & $\lambda x_9$ & $-x_{10}$& 0 & $x_9$ \\
                            \hline
                            $x_4$ & $-x_{10}$ & $-\lambda x_9$ & $x_9$ & 0 \\
                            \hline
                            \end{tabular}
                            & $\lambda \neq 0$\\
 \hline
 \end{tabular}  }
\end{table}

\bigskip
{\bf Acknowledgment.} This research is funded by University of Economics and Law, Vietnam National University Ho Chi Minh City / VNU-HCM.
%, under grant number . 
A part of this paper was done during the visit of Hieu V. Ha and Vu A. Le to Vietnam Institute for Advanced Study in Mathematics (VIASM) in summer 2022. They are very grateful to VIASM for the support and hospitality.  
% --------------------- BIBLIOGRAPHY ---------------------------------
% --------------------------------------------------------------------

\end{document}